\documentclass[12pt,reqno,a4paper]{amsart}
\usepackage{fullpage}
\usepackage[utf8]{inputenc}
\usepackage[T1]{fontenc}
\usepackage{hyperref} 
\usepackage{enumitem}
\usepackage{amsmath,amssymb,amsthm,color}
\usepackage{todonotes,framed}
\usepackage{nameref}
\usepackage{enumitem}
\usepackage{mathtools}

\title{Optimal convergence rates for goal-oriented FEM with quadratic goal functional}
\author{Roland Becker}
\address{Université de Pau et des Pays de l’Adour, IPRA-LMAP, Avenue de l’Université BP 1155, 64013 PAU Cedex, France}
\email{roland.becker@univ-pau.fr}
\author{Michael Innerberger}
\address{TU Wien, Institute of Analysis and Scientific Computing, Wiedner Hauptstr. 8-10/E101/4, 1040 Vienna, Austria}
\email{michael.innerberger@asc.tuwien.ac.at \quad \rm (corresponding author)}
\author{Dirk Praetorius}
\address{TU Wien, Institute of Analysis and Scientific Computing, Wiedner Hauptstr. 8-10/E101/4, 1040 Vienna, Austria}
\email{dirk.praetorius@asc.tuwien.ac.at}

\thanks{{\bf Acknowledgment.} The authors thankfully acknowledge support by the Austrian Science Fund (FWF) through the doctoral school \emph{Dissipation and dispersion in nonlinear PDEs} (grant W1245), the SFB \emph{Taming complexity in partial differential systems} (grant SFB F65), and the stand-alone project \emph{Computational nonlinear PDEs} (grant P33216).}

\makeatother

\def\set#1#2{\big\{#1 \,:\, #2 \big\}}

\def\eps{\varepsilon}

\usepackage{fancyhdr}
\advance\footskip0.4cm
\advance\textheight-0.4cm
\calclayout
\newcommand*\patchAmsMathEnvironmentForLineno[1]{%
  \expandafter\let\csname old#1\expandafter\endcsname\csname #1\endcsname
  \expandafter\let\csname oldend#1\expandafter\endcsname\csname end#1\endcsname
  \renewenvironment{#1}%
     {\linenomath\csname old#1\endcsname}%
     {\csname oldend#1\endcsname\endlinenomath}}%
\newcommand*\patchBothAmsMathEnvironmentsForLineno[1]{%
  \patchAmsMathEnvironmentForLineno{#1}%
  \patchAmsMathEnvironmentForLineno{#1*}}%
\AtBeginDocument{%
\patchBothAmsMathEnvironmentsForLineno{equation}%
\patchBothAmsMathEnvironmentsForLineno{align}%
\patchBothAmsMathEnvironmentsForLineno{flalign}%
\patchBothAmsMathEnvironmentsForLineno{alignat}%
\patchBothAmsMathEnvironmentsForLineno{gather}%
\patchBothAmsMathEnvironmentsForLineno{multline}%
}
\usepackage[mathlines]{lineno}

\makeatletter
\def\@seccntformat#1{\hspace*{4mm}%
  \protect\textup{\protect\@secnumfont
    \ifnum\pdfstrcmp{subsection}{#1}=0 \bfseries\fi
    \csname the#1\endcsname
    \protect\@secnumpunct
  }%
}
\makeatother

\def\tmp#1{\normalsize#1\small}

\renewcommand{\section}[2][]{%
  \vskip4mm
  \refstepcounter{section}%
  \begin{center}\bf%
    \uppercase{\thesection.~\tmp#2\normalsize}%
   \end{center}%
   \vskip2mm
}

\def\coarse{H}
\def\fine{h}

\def\N{\mathbb{N}}
\def\T{\mathbb{T}}
\def\MM{\mathcal{M}}
\def\TT{\mathcal{T}}
\def\XX{\mathcal{X}}
\def\CC{\mathcal{C}}

\def\Crel{C_{\rm rel}}

\def\Cstab{C_{\rm stab}}
\def\qred{q_{\rm red}}
\def\Clin{C_{\rm lin}}
\def\qlin{q_{\rm lin}}
\def\qctr{q_{\rm ctr}}

\def\Cnvb{C_{\rm mesh}}

\def\Ctwo{C_{\rm aux}}
\def\Cmon{C_{\rm mon}}
\def\Cdrel{C_{\rm drel}}

\def\refine{{\tt refine}}

\def\enorm#1{|\!|\!| #1 |\!|\!|}
\def\reff#1#2{\stackrel{\eqref{#1}}{#2}}
\def\refp#1#2{\stackrel{\phantom{\eqref{#1}}}{#2}}

\def\dual#1#2{\langle#1\,,\,#2\rangle}
\newcommand{\jump}[1]{[\![#1]\!]}
\newcommand{\normalvec}{\boldsymbol{n}}

\def\Cmark{C_{\rm mark}}
\def\Copt{C_{\rm opt}}
\def\norm#1#2{\|#1\|_{#2}}
\def\OO{\mathcal{O}}
\def\R{\mathbb{R}}
\def\d#1{\,{\rm d}#1}

\def\div{{\rm div}\,}

\def\UU{\mathcal{U}}

\def\RR{\mathcal{R}}

\newcounter{statement}
\newenvironment{statement}[2][!]{%
\vskip3mm
\hrule
\hrule
\hrule
\vskip1mm
\noindent%
\refstepcounter{statement}%
\bf#2~\thestatement%
\ifthenelse{\equal{#1}{!}}{.\ }{~(#1).\ }%
\it%
}{%
\vskip1mm
\hrule
\hrule
\hrule
\vskip2mm
}

\newcounter{algorithm}
\renewcommand{\thealgorithm}{\Alph{algorithm}}
\newenvironment{algorithm}[1][!]{%
	\vskip3mm
	\hrule
	\hrule
	\hrule
	\vskip1mm
	\noindent%
	\refstepcounter{algorithm}%
	\bf%
	Algorithm~\thealgorithm%
	\ifthenelse{\equal{#1}{!}}{.\ }{~(#1).\ }%
	\it%
}{%
	\vskip1mm
	\hrule
	\hrule
	\hrule
	\vskip2mm
}

\newenvironment{theorem}[1][!]{\begin{statement}[#1]{Theorem}}{\end{statement}}
\newenvironment{lemma}[1][!]{\begin{statement}[#1]{Lemma}}{\end{statement}}

\newenvironment{proposition}[1][!]{\begin{statement}[#1]{Proposition}}{\end{statement}}

\newenvironment{remark}[1][!]{\begin{statement}[#1]{Remark}}{\end{statement}}

\begin{document}

\begin{abstract}
We consider a linear elliptic PDE and a quadratic goal functional. The goal-oriented adaptive FEM algorithm (GOAFEM) solves the primal as well as a dual problem, where the goal functional is always linearized around the discrete primal solution at hand. We show that the marking strategy proposed in [Feischl et al, SIAM J.\ Numer.\ Anal., 54 (2016)] for a linear goal functional is also optimal for quadratic goal functionals, i.e., GOAFEM leads to linear convergence with optimal convergence rates.
\end{abstract}

\maketitle
\thispagestyle{fancy}

\def\A{\boldsymbol{A}}
\def\b{\boldsymbol{b}}
\def\f{\boldsymbol{f}}
\def\SS{\mathcal{S}}
\def\KK{\mathcal{K}}



\def\g{\boldsymbol{g}}
\section{Introduction}
Let $\Omega \subset \R^d$, $d \ge 2$, be a bounded Lipschitz domain. For given $f \in L^2(\Omega)$ and $\f \in [L^2(\Omega)]^d$, 
we consider a general linear elliptic partial differential equation
\begin{align}\label{eq:strongform}
\begin{split}
 -\div \A \nabla u + \b \cdot \nabla u 
 + cu &= f + \div \f  \quad \text{in } \Omega,\\
 u &= 0 \hspace{19mm} \text{on } \Gamma := \partial\Omega,
\end{split}
\end{align}
where $\A(x) \in \R^{d \times d}_{\rm sym}$ is a symmetric matrix, $\b(x) \in \R^d$, and $c(x) \in \R$. As usual, we assume that $\A, \b, c \in L^\infty(\Omega)$, that $\A$ is uniformly positive definite and that the weak form (see~\eqref{eq:primal:formulation} below) fits into the setting of the Lax--Milgram lemma.

While standard adaptivity aims to approximate the exact solution $u \in H^1_0(\Omega)$ at optimal rate in the energy norm (see, e.g., \cite{doerfler1996,mns2000,bdd2004,stevenson2007,ckns2008} for some seminal contributions and \cite{ffp2014} for the present model problem), goal-oriented adaptivity aims to approximate, at optimal rate, only the functional value $G(u) \in \R$ (also called \emph{quantity of interest} in the literature). Usually, goal-oriented adaptivity is more important in practice than standard adaptivity and has therefore attracted much interest also in the mathematical literature; see, e.g.,~\cite{MR1960405,MR2009692,MR1352472,MR2009374} for some prominent contributions. Unlike standard adaptivity, there are only few works, which aim for a mathematical understanding of optimal rates for goal-oriented adaptivity; see~\cite{ms2009,bet2011,ffghp2016,fpz2014}. While the latter works consider only \emph{linear} goal functionals, the present work aims to address, for the first time, optimal convergence rates for goal-oriented adaptivity with a \emph{nonlinear} goal functional.
More precisely, we assume that our primary interest is not in the unknown solution $u \in H^1_0(\Omega)$, but only in the functional value
\begin{align}\label{eq:goalfunctional}
 G(u) := \dual{\KK u}{u}_{H^{-1}\times H^1_0}
\end{align}
with a quadratic goal functional stemming from a bounded linear operator $\KK : H^1_0(\Omega) \to H^{-1}(\Omega)$.
Here, $\dual{\cdot}{\cdot}_{H^{-1}\times H^1_0}$ denotes the duality between $H^1_0(\Omega)$ and its dual space $H^{-1}(\Omega) := H^1_0(\Omega)'$ with respect to the extended $L^2$-scalar product. Possible examples for such goal functionals include,
e.g., $G(u) = \int_\Omega g(x) u(x)^2 \d{x}$ for a given weight $g \in L^\infty(\Omega)$ or $G(u) = \int_\Omega u(x) \, \g(x)\cdot \nabla u(x)  \d{x}$ for a given weight $\g \in [L^\infty(\Omega)]^d$.

In this work, we formulate a goal-oriented adaptive finite element method (GOAFEM), where the quantity of interest $G(u)$ is approximated by $G(u_\ell)$ for some FEM solution  $u_\ell \approx u$ such that
\begin{equation}
\label{eq:convergence}
 | G(u) - G(u_\ell) | \xrightarrow{\ell \to \infty} 0.
\end{equation}
Moreover, if $\KK$ is compact, then the convergence~\eqref{eq:convergence} holds even with optimal algebraic rates. 
Compared to the available literature on convergence of goal-adaptive FEM for linear goal functionals~\cite{ms2009,bet2011,ffghp2016,fpz2014}, we need to linearize the \emph{nonlinear} goal functional in each step of the adaptive algorithm around the present discrete solution $u_\ell$. Put in explicit terms, the prior works consider dual problems which are independent of $\ell$, while in the present work the dual problems change in each step of the adaptive loop. It is our main contribution that the additional linearization error is thoroughly taken into account for the convergence analysis.

{\bf Outline.} 
The paper is organized as follows:
Section~\ref{sec:main} formulates the finite element discretization together with two goal-oriented adaptive algorithms (Algorithm~\ref{algorithm} and \ref{algorithm:combined}).
Moreover, we state our main results:
Proposition~\ref{prop:convergenceA} and Proposition~\ref{prop:convergenceB} guarantee convergence of Algorithm~\ref{algorithm} and Algorithm~\ref{algorithm:combined}, respectively.
If the operator $\KK$ is even compact, then Theorem~\ref{theorem:main} yields linear convergence with optimal rates for Algorithm~\ref{algorithm}, while Theorem~\ref{theorem:combined} yields convergence with almost optimal rates for the simpler Algorithm~\ref{algorithm:combined}.
Section~\ref{sec:numerics} provides some numerical experiments which underline the theoretical predictions.
Sections~\ref{sec:auxiliary}--\ref{sec:proof:combined} are concerned with the proofs of our main results.


\section{Adaptive algorithm \& main result}\label{sec:main}

\subsection{Variational formulation}

Define the bilinear form 
\begin{align}\label{eq:bilinear_form}
 a(u, v) 
 := \int_\Omega \A \nabla u \cdot \nabla v \d{x} + \int_\Omega \b \cdot \nabla u \, v \d{x} 
 + \int_\Omega cuv \d{x}.
\end{align}
We suppose that $a(\cdot,\cdot)$ fits into the setting of the Lax--Milgram lemma, i.e., $a(\cdot,\cdot)$ is continuous and elliptic on $H^1_0(\Omega)$. 
While continuity
\begin{align}\label{eq:a:continuous}
 a(u,v) \le C_{\rm cnt} \, \norm{u}{H^1(\Omega)} \norm{v}{H^1(\Omega)}
 \quad \text{for all } u, v \in H^1_0(\Omega)
\end{align}
follows from the assumptions made with $C_{\rm cnt} = \norm{\A}{L^\infty(\Omega)} + \norm{\b}{L^\infty(\Omega)} + \norm{c}{L^\infty(\Omega)}$, the ellipticity 
\begin{align}\label{eq:a:elliptic}
 a(u,u) \ge C_{\rm ell} \, \norm{u}{H^1(\Omega)}^2
 \quad \text{for all } u \in H^1_0(\Omega)
\end{align}
requires additional assumptions on the coefficients, e.g.,
\begin{align*}
 \inf_{x \in \Omega} \inf_{\boldsymbol{y} \in \R^d \backslash \{0\}} \,  
 \frac{\boldsymbol{y} \cdot \A(x) \boldsymbol{y}}{|\boldsymbol{y}|^2} > 0
 \text{\ \ \ and\ \ \ }
 \b \in \boldsymbol{H}({\rm div};\Omega) \text{\ \ with\ \ }
 \inf_{x \in \Omega} \Big( \frac{1}{2} \, \div \b(x) + c(x) \Big) \ge 0.
\end{align*}%
The weak formulation of~\eqref{eq:strongform} reads
\begin{align}\label{eq:primal:formulation}
 a(u, v) = F(v) := \int_\Omega fv \d{x} -\int_\Omega \f \cdot \nabla v \d{x}
 \quad \text{for all } v \in H^1_0(\Omega).
\end{align}
According to the Lax--Milgram lemma, \eqref{eq:primal:formulation} admits a unique solution $u \in H^1_0(\Omega)$. Given $w \in H^1_0(\Omega)$, the same argument applies and proves that the (linearized) dual problem 
\begin{align}\label{eq:dual:formulation}
 a(v, z[w]) = b(v,w) + b(w,v)
 \quad \text{for all } v \in H^1_0(\Omega)
\end{align}
admits a unique solution $z[w] \in H^1_0(\Omega)$, where we abbreviate the notation by use of $b(v, w) := \dual{\KK v}{w}_{H^{-1}\times H^1_0}$. We note that $b(\cdot,\cdot)$ is, in particular, a continuous bilinear form on $H^1_0(\Omega)$.
Throughout, we denote by $\enorm{v}^2 := \int_\Omega \A \nabla v \cdot \nabla v \d{x}$ the energy norm induced by the principal part of $a(\cdot,\cdot)$, which is an equivalent norm on $H^1_0(\Omega)$.
Finally, we stress that all main results also apply to the case that $a(\cdot,\cdot)$ satisfies only a G\r{a}rding inequality (instead of the strong ellipticity~\eqref{eq:a:elliptic})
as long as the weak formulations~\eqref{eq:primal:formulation} and~\eqref{eq:dual:formulation} are well-posed; see Section~\ref{section:garding} below.

\subsection{Finite element method}\label{subsec:fem}

For a conforming triangulation  $\TT_\coarse$ of $\Omega$ into compact simplices and a polynomial degree $p \ge 1$, we consider the conforming finite element space
\begin{align}
 \XX_\coarse := \set{v_\coarse \in H^1_0(\Omega)}{\forall T \in \TT_\coarse \quad v_\coarse|_T \text{ is a polynomial of degree } \le p}.
\end{align}
We approximate  $u \approx u_\coarse \in \XX_\coarse$ and $z[w] \approx z_\coarse[w] \in \XX_\coarse$. 
More precisely, the Lax--Milgram lemma yields the existence and uniqueness of discrete FEM solutions $u_\coarse, z_\coarse[w] \in \XX_\coarse$ of
\begin{align}\label{eq:discrete:formulation}
 a(u_\coarse, v_\coarse) = F(v_\coarse)
 \,\,\, \text{and} \,\,\,
 a(v_\coarse, z_\coarse[w]) =  b(v_\coarse, w) + b(w, v_\coarse)
 \,\,\,  \text{for all } v_\coarse \in \XX_\coarse.
\end{align}

\subsection{Linearization of the goal functional}

To control the goal error $|G(u) - G(u_\coarse)|$, we employ the dual problem. Note that
\begin{align*}
 b(u - u_\coarse, u - u_\coarse)
 &= b(u,u) - b(u_\coarse,u) - b(u,u_\coarse) + b(u_\coarse,u_\coarse)
 \\&
 = \big[ G(u) - G(u_\coarse) \big] - \big[ b(u_\coarse,u) + b(u,u_\coarse) -2 b(u_\coarse,u_\coarse) \big] 
 \\&
 = \big[ G(u) - G(u_\coarse) \big] - \big[ b(u_\coarse, u-u_\coarse) + b(u - u_\coarse,u_\coarse) \big]. 
\end{align*}%
With the dual problem and the Galerkin orthogonality, we rewrite the second bracket as
\begin{align*}
 b(u - u_\coarse,u_\coarse) + b(u_\coarse, u - u_\coarse)
 \reff{eq:dual:formulation}= a(u - u_\coarse, z[u_\coarse]) 
 = a(u - u_\coarse, z[u_\coarse] - z_\coarse[u_\coarse]).
\end{align*}
With continuity of the bilinear forms $a(\cdot, \cdot)$ and $b(\cdot, \cdot)$, we thus obtain that
\begin{align}\label{eq:ansart:aposteriori}
 \begin{split}
 \big| G(u) - G(u_\coarse) \big|
 &= \big| a(u - u_\coarse, z[u_\coarse] - z_\coarse[u_\coarse])  + b(u - u_\coarse, u - u_\coarse) \big|
 \\&
 \lesssim \enorm{u - u_\coarse} \, \big[ \, \enorm{z[u_\coarse] - z_\coarse[u_\coarse]} + \enorm{u - u_\coarse} \, \big].
 \end{split}
\end{align}

\subsection{Mesh-refinement}

Let $\TT_0$ be a given conforming triangulation of $\Omega$.
We suppose that the mesh-refinement is a deterministic and fixed strategy, e.g., newest vertex bisection~\cite{stevenson2008}. 
For each triangulation $\TT_\coarse$ and marked elements $\MM_\coarse \subseteq \TT_\coarse$, let $\TT_\fine := \refine(\TT_\coarse,\MM_\coarse)$ be the coarsest triangulation, where all $T \in \MM_\coarse$ have been refined, i.e., $\MM_\coarse \subseteq \TT_\coarse \backslash \TT_\fine$. 
We write $\TT_\fine \in \T(\TT_\coarse)$, if $\TT_\fine$ results from $\TT_\coarse$ by finitely many steps of refinement.
To abbreviate notation, let $\T:=\T(\TT_0)$.

We further suppose that each refined element has at least two sons, i.e.,
\begin{align}\label{eq:mesh-sons}
	\#(\TT_\coarse \backslash \TT_\fine) + \#\TT_\coarse \le \#\TT_\fine
	\quad \text{for all } \TT_\coarse \in \T \text{ and all } \TT_\fine \in \T(\TT_\coarse),
\end{align}
and that the refinement rule satisfies the mesh-closure estimate
\begin{align}\label{eq:mesh-closure}
	\#\TT_\ell - \#\TT_0 \le \Cnvb\,\sum_{j=0}^{\ell-1}\#\MM_j
	\quad \text{for all } \ell \in \N,
\end{align}
where $\Cnvb>0$ depends only on $\TT_0$.
This has first been proved for 2D newest vertex bisection in~\cite{bdd2004} and has later been generalized to arbitrary dimension $d\ge2$ in~\cite{stevenson2008}.
While both works require an additional admissibility assumption on $\TT_0$, this has been proved unnecessary at least for 2D in~\cite{kpp2013}.
Finally, it has been proved in~\cite{ckns2008,stevenson2007} that newest vertex bisection ensures the overlay estimate, i.e., for all triangulations $\TT_\coarse,\TT_\fine \in \T$, there exists a common refinement $\TT_\coarse \oplus \TT_\fine \in \T(\TT_\coarse) \cap \T(\TT_\fine)$ which satisfies that
\begin{align}\label{eq:mesh-overlay}
	\#(\TT_\coarse \oplus \TT_\fine) \le \#\TT_\coarse + \#\TT_\fine - \#\TT_0.
\end{align}
For meshes with first-order hanging nodes, \eqref{eq:mesh-sons}--\eqref{eq:mesh-overlay} are analyzed in~\cite{bn2010}, while T-splines and hierarchical splines for isogeometric analysis are considered in~\cite{mp2015,morgenstern2016} and \cite{bgmp2016,ghp2017}, respectively.

\subsection{Error estimators}

For $\TT_\coarse \in \T$ and $v_\coarse \in \XX_\coarse$, let
\begin{equation*}
	\eta_\coarse(T, v_\coarse) \geq 0
	\quad \text{and} \quad
	\zeta_\coarse(T, v_\coarse) \geq 0
	\quad \text{for all } T \in \TT_\coarse
\end{equation*}
be given refinement indicators. For $\UU_\coarse \subseteq \TT_\coarse$, let
\begin{equation*}
	\eta_\coarse(\UU_\coarse, v_\coarse)
	:=
	\Big( \sum_{T \in \UU_\coarse} \eta_\coarse(T, v_\coarse)^2 \Big)^{1/2}
	\quad \text{and} \quad
	\zeta_\coarse(\UU_\coarse, v_\coarse)
	:=
	\Big( \sum_{T \in \UU_\coarse} \zeta_\coarse(T, v_\coarse)^2 \Big)^{1/2}.
\end{equation*}
To abbreviate notation, let $\eta_\coarse(v_\coarse) := \eta_\coarse(\TT_\coarse, v_\coarse)$ and $\zeta_\coarse(v_\coarse) := \zeta_\coarse(\TT_\coarse, v_\coarse)$.

We suppose that the estimators $\eta_\coarse$ and $\zeta_\coarse$ satisfy the following \emph{axioms of adaptivity} from~\cite{axioms}: There exist constants $\Cstab, \Crel, \Cdrel > 0$ and $0 < \qred < 1$ such that for all $\TT_\coarse \in \T$ and all $\TT_\fine \in \T(\TT_\coarse)$, the following assumptions are satisfied:
\renewcommand{\theenumi}{{A\arabic{enumi}}}
\begin{enumerate}
	\bf
	\item\label{assumption:stab} stability: \rm 
	For all $v_\fine \in \XX_\fine$, $v_\coarse \in \XX_\coarse$, and $\UU_\coarse \subseteq \TT_\fine \cap \TT_\coarse$, it holds that
	\begin{align*}
	\big| \eta_\fine(\UU_\coarse, v_\fine) - \eta_\coarse(\UU_\coarse, v_\coarse) \big|
	+ \big| \zeta_\fine(\UU_\coarse, v_\fine) - \zeta_\coarse(\UU_\coarse, v_\coarse) \big|
	&\le \Cstab \, \enorm{v_\fine - v_\coarse}.
	\end{align*}
	\bf
	\item\label{assumption:red} reduction: \rm
	For all $v_\coarse \in \XX_\coarse$, it holds that
	\begin{align*}
	\eta_\fine(\TT_\fine \backslash \TT_\coarse, v_\coarse) 
	&\le \qred \, \eta_\coarse(\TT_\coarse \backslash \TT_\fine, v_\coarse),
	\\
	\zeta_\fine(\TT_\fine \backslash \TT_\coarse, v_\coarse) 
	&\le \qred \, \zeta_\coarse(\TT_\coarse \backslash \TT_\fine, v_\coarse).
	\end{align*}
	\bf
	\item\label{assumption:rel} reliability: \rm
	For all $w \in H^1_0(\Omega)$, the Galerkin solutions $u_\coarse, z_\coarse[w] \in \XX_\coarse$ to~\eqref{eq:discrete:formulation} satisfy that
	\begin{align*}
	\enorm{u - u_\coarse} 
	&\le \Crel \, \eta_\coarse(u_\coarse),
	\\
	\enorm{z[w] - z_\coarse[w]} 
	&\le \Crel \, \zeta_\coarse(z_\coarse[w]).
	\end{align*} 
	\bf
	\item\label{assumption:drel} discrete reliability: \rm
	For all $w \in H^1_0(\Omega)$, the Galerkin solutions $u_\coarse, z_\coarse[w] \in \XX_\coarse$ and $u_\fine, z_\fine[w] \in \XX_\fine$ to~\eqref{eq:discrete:formulation} satisfy that
	\begin{align*}
	\enorm{u_\fine - u_\coarse} 
	&\le \Cdrel \, \eta_\coarse(\TT_\coarse \backslash \TT_\fine, u_\coarse),
	\\
	\enorm{z_\fine[w] - z_\coarse[w]}
	&\le \Cdrel \, \zeta_\coarse(\TT_\coarse \backslash \TT_\fine, z_\coarse[w]). 
	\end{align*} 
\end{enumerate}

We note that the axioms \eqref{assumption:stab}--\eqref{assumption:drel} are satisfied for, e.g., standard residual error estimators. Given $w \in H^1_0(\Omega)$, the mapping $v \mapsto b(v,w) + b(w,v)$ is linear and continuous by assumption. Hence, the Riesz theorem from functional analysis guarantees the existence (and uniqueness) of $g[w] \in H^1_0(\Omega)$ such that
\begin{align}\label{eq0:dual:right-hand-side}
 (g[w],v)_{H^1} 
 :=\! \int_\Omega \! g[w] v \d{x} + \! \int_\Omega \! \nabla g[w] \cdot \nabla v \d{x}
 = b(v,w) + b(w,v)
 \text{ for all } v \in H_0^1(\Omega).
\end{align}
With $\g[w] = - \nabla g[w]$, we thus get that
\begin{equation}\label{eq:dual:right-hand-side}
	b(v,w) + b(w,v) = \int_\Omega g[w] v \d{x} -\int_\Omega \g[w] \cdot \nabla v \d{x}
	\quad\text{for all }
	v \in H_0^1(\Omega),
\end{equation}
i.e., the right-hand sides of the primal problem~\eqref{eq:primal:formulation} and the (linearized) dual problem~\eqref{eq:dual:formulation} take the same form.
With this%
\footnote{Recall the strong form of the primal problem 
$$-\div \A \nabla u + \b \cdot \nabla u  + cu = f + \div \f \quad \text{in } \Omega$$ 
and note that the corresponding (linearized) strong form of the dual problem reads 
$$-\div \A \nabla z - \b \cdot \nabla z + (c-\div\b)z = g[w] + \div \g[w] \quad \text{in } \Omega.$$%
}, the residual error estimators 
read for $v_\coarse \in \XX_\coarse$ as
\begin{align*}
	\eta_\coarse(T,v_\coarse)^2
	&:=
	h_T^2 \norm{-\div(\A \nabla v_\coarse + \f) + \b \cdot \nabla v_\coarse + c v_\coarse - f}{L^2(T)}^2\\
	&\hspace{100pt} + h_T \norm{\jump{(\A \nabla v_\coarse + \f) \cdot \normalvec}}{L^2(\partial T \cap \Omega)}^2,\\
	\zeta_\coarse(T,v_\coarse)^2
	&:=
	h_T^2 \norm{-\div(\A \nabla v_\coarse + \g[u_\coarse]) - \b \cdot \nabla v_\coarse + (c-\div\b) v_\coarse - g[u_\coarse]}{L^2(T)}^2\\
	&\hspace{100pt} + h_T \norm{\jump{(\A \nabla v_\coarse + \g[u_\coarse]) \cdot \normalvec}}{L^2(\partial T \cap \Omega)}^2,
\end{align*}
where $\jump{\cdot}$ denotes the jump across edges and $\normalvec$ is the outwards-facing unit normal vector. We stress that our experiments below directly provide $g[w] \in L^2(\Omega)$ and $\g[w] \in [L^2(\Omega)]^d$ satisfying the representation~\eqref{eq:dual:right-hand-side}, so that there is, in fact, no need to solve~\eqref{eq0:dual:right-hand-side}.

\subsection{Adaptive algorithm}
\quad We consider the following adaptive algorithm, which adapts the marking strategy proposed in~\cite{fpz2014}.

\begin{algorithm}\label{algorithm}
	{\bfseries Input:} Adaptivity parameters $0 < \theta \le 1$ and $\Cmark \ge 1$, initial mesh $\TT_0$.
	\newline
	{\bfseries Loop:} For all $\ell = 0,1,2,\dots$, perform the following steps~{\rm(i)--(v)}:
	\begin{itemize}
		\item[\rm(i)] Compute the discrete solutions $u_\ell, z_\ell[u_\ell] \in \XX_\ell$ to~\eqref{eq:discrete:formulation}.
		\item[\rm(ii)] Compute the refinement indicators $\eta_\ell(T,u_\ell)$ and $\zeta_\ell(T, z_\ell[u_\ell])$ for all $T \in \TT_\ell$.
		\item[\rm(iii)] Determine sets $\overline\MM_\ell^{u}, \overline\MM_\ell^{uz} \subseteq \TT_\ell$ of up to the multiplicative constant $\Cmark$ minimal cardinality such that
		\begin{subequations}\label{eq:doerfler}
			\begin{align}\label{eq:doerfler:u}
			\theta \, \eta_\ell(u_\ell)^2 
			&\le \eta_\ell(\overline\MM_\ell^{u},u_\ell)^2,
			\\ \label{eq:doerfler:uz}
			\theta \, \big[ \, \eta_\ell(u_\ell)^2 + \zeta_\ell(z_\ell[u_\ell])^2 \, \big] 
			&\le \big[ \, \eta_\ell(\overline\MM_\ell^{uz},u_\ell)^2 + \zeta_\ell(\overline\MM_\ell^{uz},z_\ell[u_\ell])^2 \, \big].
			\end{align} 
		\end{subequations}
		\item[\rm(iv)] Let $\MM_\ell^{u} \subseteq \overline\MM_\ell^{u}$ and $\MM_\ell^{uz} \subseteq \overline\MM_\ell^{uz}$ with $\#\MM_\ell^{u} =  \#\MM_\ell^{uz} =  \min\{ \, \#\overline\MM_\ell^{u} \,,\, \#\overline\MM_\ell^{uz} \, \}$.
		\item[\rm(v)] Define $\MM_\ell := \MM_\ell^{u} \cup \MM_\ell^{uz}$ and generate $\TT_{\ell+1} := \refine(\TT_\ell, \MM_\ell)$.
	\end{itemize}
	{\bfseries Output:} Sequence of triangulations $\TT_\ell$ with corresponding discrete solutions $u_\ell$ and $z_\ell[u_\ell]$ as well as error estimators $\eta_\ell(u_\ell)$ and $\zeta_\ell(z_\ell[u_\ell])$.
\end{algorithm}

With Algorithm~\ref{algorithm:combined} below, we give and examine an alternative adaptive algorithm that is seemingly cheaper in computational costs.

Our first result states that Algorithm~\ref{algorithm} indeed leads to convergence.
\begin{proposition}
\label{prop:convergenceA}
	For any bounded linear operator $\KK : H_0^1(\Omega) \to H^{-1}(\Omega)$, there hold the following statements~{\rm(i)}--{\rm(ii)}:
		
	{\rm(i)} There exists a constant $\Crel^\prime > 0$ such that
	\begin{equation}\label{eq:goal-upper-bound}
		\big| G(u) - G(u_\coarse) \big|
		\le
		\Crel^\prime \, \eta_\coarse(u_\coarse) \, \big[ \, \eta_\coarse(u_\coarse)^2 + \zeta_\coarse(z_\coarse[u_\coarse])^2 \, \big]^{1/2}
		\quad \text{for all } \TT_\coarse \in \T.
	\end{equation}

	{\rm(ii)} For all $0 < \theta \leq 1$ and $1 < \Cmark \leq \infty$, Algorithm~\ref{algorithm} leads to convergence
	\begin{equation}
	\label{eq:plain-convergenceA}
		| G(u) - G(u_\ell) | 
		\le
		\Crel^\prime \eta_\ell(u_\ell) \, \big[ \,
			\eta_\ell(u_\ell)^2 + \zeta_\ell(z_\ell[u_\ell])^2 \,
		\big]^{1/2}
		\longrightarrow 0
		\quad\text{as }
		\ell \to \infty.
	\end{equation}
The constant $\Crel^\prime$ depends only on the constants from~\eqref{assumption:stab}--\eqref{assumption:rel}, the bilinear form $a(\cdot,\cdot)$ and the boundedness of $\KK$.
\end{proposition}

To formulate our main result on optimal convergence rates, we need some additional notation.
For $N \in \N_0$, let $\T_N:=\set{\TT\in\T}{\#\TT-\#\TT_0\le N}$ denote the (finite) set of all refinements of $\TT_0$, which have at most $N$ elements more
than $\TT_0$. 
For $s,t>0$, we define
\begin{align*}
	\norm{u}{\mathbb{A}_s} 
	&:= \sup_{N\in\N_0} \Big((N+1)^s \min_{\TT_\coarse\in\T_N} \eta_\coarse(u_\coarse) \Big) \in \R_{\geq 0} \cup \{ \infty \},\\
	\norm{z[u]}{\mathbb{A}_t} 
	&:= \sup_{N\in\N_0} \Big((N+1)^t\min_{\TT_\coarse\in\T_N} \zeta_\coarse(z_\coarse[u]) \Big) \in \R_{\geq 0} \cup \{ \infty \}.
\end{align*}
In explicit terms, e.g., $\norm{u}{\mathbb{A}_s} < \infty$ means that an algebraic convergence rate $\OO(N^{-s})$ for the error estimator $\eta_\ell$ is possible, if the optimal triangulations are chosen.

The following theorem concludes the main results of the present work:

\begin{theorem}\label{theorem:main}
For any compact operator $\KK : H_0^1(\Omega) \to H^{-1}(\Omega)$, there even hold the following statements {\rm(i)}--{\rm(ii)}, which improve Proposition~\ref{prop:convergenceA}{\rm(ii)}:

{\rm(i)} For all $0 < \theta \le 1$ and $\Cmark \ge 1$, there exists $\ell_0 \in \N_0$, $\Clin > 0$, and $0 < \qlin < 1$ such that Algorithm~\ref{algorithm} guarantees that, for all $\ell, n \in \N_0$ with $n \ge \ell \ge \ell_0$,
\begin{align}\label{eq:linear}
 \eta_n(u_n) \, \big[ \eta_n(u_n)^2 + \zeta_n(z_n[u_n])^2 \, \big]^{1/2}
 \le \Clin \qlin^{n-\ell} \, \eta_\ell(u_\ell) \, \big[ \eta_\ell(u_\ell)^2 + \zeta_\ell(z_\ell[u_\ell])^2 \, \big]^{1/2}.
\end{align}

{\rm(ii)} There exist $\Copt > 0$ and $\ell_0 \in \N_0$ such that Algorithm~\ref{algorithm} guarantees that, for all $0 < \theta < \theta_{\rm opt} := (1+\Cstab^2\Cdrel^2)^{-1}$, for all $s,t > 0$ with $\norm{u}{\mathbb{A}_s} + \norm{z[u]}{\mathbb{A}_t} < \infty$, and all $\ell \in \N_0$ with $\ell \geq \ell_0$, it holds that
\begin{align}\label{eq:optimal}
	\eta_\ell(u_\ell)\big[ \, \eta_\ell(u_\ell)^2 + \zeta_\ell(z_\ell[u_\ell])^2 \, \big]^{1/2}
	\le 
	\Copt \, \norm{u}{\mathbb{A}_s}(\norm{u}{\mathbb{A}_s}+\norm{z}{\mathbb{A}_t})\,(\#\TT_\ell-\#\TT_0)^{-\alpha},
\end{align}
where $\alpha := \min\{2s,s+t\}$.

The constants $\Clin$, $\qlin$, and $\ell_0$ depend only on $\theta$, $\qred$, $\Cstab$, $\Crel$, the bilinear form~$a(\cdot,\cdot)$, and the compact operator $\KK$.
The constant $\Copt$ depends only on $\theta$, $\Cnvb$, $\Cmark$, $\Clin$, $\qlin$, $\ell_0$, and \eqref{assumption:stab}--\eqref{assumption:drel}.
\end{theorem}

\begin{remark}\label{rem:algorithm}
{\rm(i)} We note that, according to the considered dual problem~\eqref{eq:dual:formulation}, the goal functional~\eqref{eq:goalfunctional} is linearized around $u_\ell$ in each step of the adaptive algorithm. Hence, we must enforce that the linearization error satisfies that $\enorm{z_\ell[u]-z_\ell[u_\ell]} \to 0$ as $\ell \to \infty$. This is guaranteed by Proposition~\ref{prop:convergenceA}{\rm(ii)} and Theorem~\ref{theorem:main}{\rm(i)}, since both factors of the product involve the primal error estimator $\eta_\ell(u_\ell)$.

{\rm(ii)} For a linear goal functional and hence $z_\ell[u] = z_\ell[u_\ell]$, the work~\cite{fpz2014} considers plain $\zeta_\ell^2$ (instead of $\eta_\ell^2 + \zeta_\ell^2$) for the D\"orfler marking~\eqref{eq:doerfler:uz} and then proves a convergence behavior $|G(u)-G(u_\ell)| \lesssim \eta_\ell \zeta_\ell = \OO\big((\#\TT_\ell)^{-\alpha}\big)$ for the estimator product, where $\alpha = s + t$ with $s > 0$ being the optimal rate for the primal problem and $t > 0$ being the optimal rate for the dual problem. Instead, Algorithm~\ref{algorithm} will only lead to $\OO\big((\#\TT_\ell)^{-\alpha}\big)$, where $\alpha = \min\{2s, s+t\}$; see Theorem~\ref{theorem:main}{\rm(ii)}.

{\rm(iii)} The marking strategy proposed in~\cite{bet2011}, where D\"orfler marking is carried out for the weighted estimator
\begin{equation}\label{eq:bet1}
	\rho_\coarse(T,u_\coarse,z_\coarse[u_\coarse])^2
	:=
	\eta_\coarse(T,u_\coarse)^2
	\zeta_\coarse(z_\coarse[u_\coarse])^2
	+
	\eta_\coarse(u_\coarse)^2
	\zeta_\coarse(T,z_\coarse[u_\coarse])^2,
\end{equation}
might be unable to ensure convergence of the linearization error $\enorm{z_\ell[u]-z_\ell[u_\ell]}$, since in every step D\"orfler marking is implied for either $\eta_\coarse(u_\coarse)$ \emph{or}  $\zeta_\coarse(z_\coarse[u_\coarse])$; cf.~\cite{fpz2014}.
If one instead considers
\begin{equation}\label{eq:bet2}
\begin{split}
	\varrho_\coarse(T,u_\coarse,z_\coarse[u_\coarse])^2
	&:=
	\eta_\coarse(T,u_\coarse)^2
	\big[ \eta_\coarse(u_\coarse)^2 + \zeta_\coarse(z_\coarse[u_\coarse])^2 \big] \\
	&\qquad+
	\eta_\coarse(u_\coarse)^2
	\big[ \eta_\coarse(T,u_\coarse)^2 + \zeta_\coarse(T,z_\coarse[u_\coarse])^2 \big],
\end{split}
\end{equation}
the present results and the analysis in~\cite{fpz2014} make it clear that this strategy implies convergence with rate $\min\{2s,s+t\}$.
Details are omitted.
\end{remark}

\subsection{Alternative adaptive algorithm}
From the upper bound \eqref{eq:goal-upper-bound} in Proposition~\ref{prop:convergenceA}(i), we can further estimate the goal error by
\begin{align*}
\big| G(u) - G(u_\coarse) \big|
\leq
\Crel^\prime \, \big[ \, \eta_\coarse(u_\coarse)^2 + \zeta_\coarse(z_\coarse[u_\coarse])^2 \, \big]
\quad\text{for all }
\TT_\coarse \in \T.
\end{align*}
This suggests the following algorithm, which marks elements solely based on the combined estimator.

\begin{algorithm}\label{algorithm:combined}
	{\bfseries Input:} Adaptivity parameters $0 < \theta \le 1$ and $\Cmark \ge 1$, initial mesh $\TT_0$.
	\newline
	{\bfseries Loop:} For all $\ell = 0,1,2,\dots$, perform the following steps~{\rm(i)--(iv)}:
	\begin{itemize}
		\item[\rm(i)] Compute the discrete solutions $u_\ell, z_\ell[u_\ell] \in \XX_\ell$ to~\eqref{eq:discrete:formulation}.
		\item[\rm(ii)] Compute the refinement indicators $\eta_\ell(T,u_\ell)$ and $\zeta_\ell(T, z_\ell[u_\ell])$ for all $T \in \TT_\ell$.
		\item[\rm(iii)] Determine a set $\MM_\ell \subseteq \TT_\ell$ of up to the multiplicative constant $\Cmark$ minimal cardinality such that
		\begin{equation}\label{eq:doerfler:dual}
		\theta \, \big[ \, \eta_\ell(u_\ell)^2 + \zeta_\ell(z_\ell[u_\ell])^2 \, \big] 
		\le \big[ \, \eta_\ell(\MM_\ell,u_\ell)^2 + \zeta_\ell(\MM_\ell,z_\ell[u_\ell])^2 \, \big].
		\end{equation}
		\item[\rm(iv)] Generate $\TT_{\ell+1} := \refine(\TT_\ell, \MM_\ell)$.
	\end{itemize}
	{\bfseries Output:} Sequence of triangulations $\TT_\ell$ with corresponding discrete solutions $u_\ell$ and $z_\ell[u_\ell]$ as well as error estimators $\eta_\ell(u_\ell)$ and $\zeta_\ell(z_\ell[u_\ell])$.
\end{algorithm}

First, we note that Algorithm~\ref{algorithm:combined} also leads to convergence.

\begin{proposition}
	\label{prop:convergenceB}
	For any bounded linear operator $\KK$, there hold the following statements {\rm(i)}--{\rm(ii)}:
	
	{\rm(i)} There exists a constant $\Crel^\prime > 0$ such that
		\begin{equation}\label{eq:goal-upper-boundB}
		\big| G(u) - G(u_\coarse) \big|
		\le
		\Crel^\prime \, \big[ \, \eta_\coarse(u_\coarse)^2 + \zeta_\coarse(z_\coarse[u_\coarse])^2 \, \big]
		\quad \text{for all } \TT_\coarse \in \T.
		\end{equation}
	
	{\rm(ii)} For all $0 < \theta \leq 1$ and $1 < \Cmark \leq \infty$, Algorithm~\ref{algorithm:combined} leads to convergence
		\begin{equation}
		\label{eq:plain-convergenceB}
			| G(u) - G(u_\ell) | 
			\le
			\Crel^\prime \, \big[ \,
				\eta_\ell(u_\ell)^2 + \zeta_\ell(z_\ell[u_\ell])^2 \,
			\big]
			\longrightarrow 0
			\quad\text{as }
			\ell \to \infty.
		\end{equation}
	The constant $\Crel^\prime$ depends only on the constants from~\eqref{assumption:stab}--\eqref{assumption:rel}, the bilinear form~$a(\cdot,\cdot)$, and the boundedness of $\KK$.
\end{proposition}

The following theorem proves linear convergence of Algorithm~\ref{algorithm:combined} with almost optimal convergence rate, where we note that $\beta \leq \alpha$ for the rates in~\eqref{eq:optimal} and \eqref{eq:combined:optimal}.
By abuse of notation we use the same constants as in Theorem~\ref{theorem:main}.
 
\begin{theorem}\label{theorem:combined}
	For any compact operator $\KK$, there even hold the following statements {\rm(i)}--{\rm(ii)}, which improve Proposition~\ref{prop:convergenceB}(ii):
	
	{\rm(i)} For all $0 < \theta \le 1$ and $\Cmark \ge 1$, there exists $\ell_0 \in \N_0$, $\Clin > 0$, and $0 < \qlin < 1$ such that Algorithm~\ref{algorithm:combined} guarantees that, for all $\ell, n \in \N_0$ with $n \ge \ell \ge \ell_0$,
	\begin{align}\label{eq:combined:linear}
	\big[ \eta_n(u_n)^2 + \zeta_n(z_n[u_n])^2 \, \big]
	\leq
	\Clin \qlin^{n-\ell} \, \big[
	\eta_\ell(u_\ell)^2 + \zeta_\ell(z_\ell[u_\ell])^2
	\, \big].
	\end{align}
	
	{\rm(ii)} There exist $\Copt > 0$ and $\ell_0 \in \N_0$ such that Algorithm~\ref{algorithm} guarantees that, for all $0 < \theta < \theta_{\rm opt} := (1+\Cstab^2\Cdrel^2)^{-1}$, for all $s,t > 0$ with $\norm{u}{\mathbb{A}_s} + \norm{z[u]}{\mathbb{A}_t} < \infty$, and all $\ell \in \N_0$ with $\ell \geq \ell_0$, it holds that
	\begin{align}\label{eq:combined:optimal}
	\big[ \, \eta_\ell(u_\ell)^2 + \zeta_\ell(z_\ell[u_\ell])^2 \, \big]
	\leq
	\Copt \, (\norm{u}{\mathbb{A}_s}^2 + \norm{z}{\mathbb{A}_t}^2) \, (\#\TT_\ell-\#\TT_0)^{-\beta},
	\end{align}
	where $\beta := \min\{2s,2t\}$.
	
	The constants $\Clin$, $\qlin$, and $\ell_0$ depend only on $\theta$, $\qred$, $\Cstab$, $\Crel$, the bilinear form~$a(\cdot,\cdot)$, and the compact operator $\KK$.
	The constant $\Copt$ depends only on $\theta$, $\Cnvb$, $\Cmark$, and \eqref{assumption:stab}--\eqref{assumption:drel}.
\end{theorem}

Note that Algorithm~\ref{algorithm:combined} has slightly lower computational costs than Algorithm~\ref{algorithm}, but achieves only a lower rate in general.
However, if there holds $s \leq t$ both algorithms achieve rate $2s$.

\subsection{Extension of analysis to compactly perturbed elliptic problems}
\label{section:garding}

For the ease of presentation, we have restricted ourselves to the case that the bilinear form $a(\cdot,\cdot)$ from~\eqref{eq:bilinear_form} is continuous~\eqref{eq:a:continuous} and elliptic~\eqref{eq:a:elliptic}. Actually, it suffices to assume that $a(\cdot,\cdot)$ is continuous and that the energy norm $\enorm{\cdot}$ induced by the principal part is an equivalent norm on $H^1_0(\Omega)$, e.g., by assuming that $\A \in L^\infty(\Omega)$ is uniformly positive definite. Then, $a(\cdot,\cdot)$ is elliptic up to some compact perturbation (and hence satisfies a G\r{a}rding inequality). A prominent example for this problem class is the Helmholtz problem.

We have to assume that the primal formulation~\eqref{eq:primal:formulation} 
is well-posed, i.e., for all $w \in H^1_0(\Omega)$ it holds that
\begin{align*}
 \big[ \,
	 a(w,v) = 0 
	 ~ \text{for all } v \in H^1_0(\Omega)
 \, \big]
	 \quad\Longrightarrow\quad
	  w = 0.
\end{align*}
Then, the Fredholm alternative and standard functional analysis imply that the primal formulation~\eqref{eq:primal:formulation} as well as the dual formulation~\eqref{eq:dual:formulation} admit unique solutions. Moreover, as soon as $\TT_\coarse$ is sufficiently fine, also the FEM problems~\eqref{eq:discrete:formulation} admit unique solutions and, more importantly, the discrete inf-sup constants are uniformly bounded from below; see, e.g.,~\cite[Section~2]{bhp2017}.

As noted in~\cite{bhp2017}, such an analytical setting requires only two minor modifications of adaptive algorithms:
\begin{itemize}
\item[\rm(a)] \emph{Step~{\rm(i)} in Algorithm~\ref{algorithm} or Algorithm~\ref{algorithm:combined}:} If the discrete solutions $u_\ell$ and $z_\ell[u_\ell]$ exist (and are hence also unique), then we proceed as before. If either $u_\ell$ or $z_\ell[u_\ell]$ does not exist, then the mesh $\TT_{\ell+1}$ is obtained by uniform refinement of $\TT_\ell$, i.e., $\MM_\ell := \TT_\ell$.
\item[\rm(b)] \emph{Step~{\rm(iv)} of Algorithm~\ref{algorithm} or step~{\rm(iii)} of Algorithm~\ref{algorithm:combined}:} Having determined a set of marked elements $\MM_\ell \subseteq \TT_\ell$, we select a superset $\MM_\ell^\# \supseteq \MM_\ell$ with $\# \MM_\ell^\# \le 2 \, \#\MM_\ell$ as well as $\MM_\ell^\# \cap \set{T \in \TT_\ell}{|T| \ge |T'| \text{ for all } T' \in \TT_\ell} \neq \emptyset$ and define the refined mesh $\TT_{\ell+1} := \refine(\TT_\ell, \MM_\ell^\#)$ via the extended set of marked elements.
\end{itemize}
It is observed in~\cite{bhp2017} that uniform refinement caused by the modification~{\rm(a)} can only occur finitely many times. Moreover, the modification~{\rm(b)} ensures that $H^1_0(\Omega) = \overline{\bigcup_{\ell=0}^\infty \XX_\ell}$ so that the adaptive algorithm converges, indeed, to the right limit. For standard adaptive FEM, it is shown in~\cite{bhp2017} that this procedure still leads to optimal convergence rates. We note that the arguments from~\cite{bhp2017} obviously extend to the present goal-oriented adaptive FEM.


\section{Numerical experiments}\label{sec:numerics}

In this section, we underline our theoretical findings by some numerical examples.
As starting point of all examples, we use equation~\eqref{eq:strongform} with $\A = I$, $\b = \boldsymbol{0}$, and $c=0$ on the unit square $\Omega = (0,1)^2$.
The initial mesh $\TT_0$ on $\Omega$ is obtained from certain uniform refinements from the mesh shown in Figure~\ref{fig:mesh}.
All examples are computed with conforming finite elements of order $p=1$ and $p=2$, as outlined in Section~\ref{subsec:fem}.

In the following, we consider the marking strategies of Algorithm~\ref{algorithm} and Algorithm~\ref{algorithm:combined} (denoted by A and B, respectively), as well as the marking strategies outlined in Remark~\ref{rem:algorithm}(iii), i.e., D\"orfler marking for \eqref{eq:bet1} and \eqref{eq:bet2}, which will be denoted by BET1 and BET2, respectively.
If not stated otherwise, the marking parameter is $\theta = 0.5$ for all experiments.

\begin{figure}
	\centering
	\includegraphics[width=0.7\linewidth]{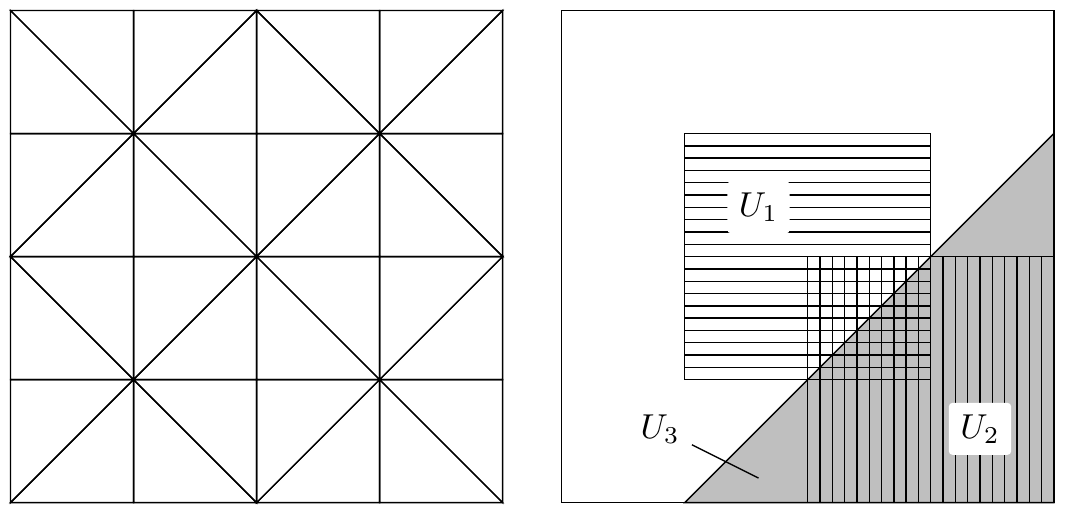}
	\caption{Initial mesh (left) and sets $U_1, U_2, U_3$ (right) on the unit square $\Omega = (0,1)^2$.}
	\label{fig:mesh}
\end{figure}

\subsection{Weighted $L^2$ norm}\label{subsec:L2}

Suppose some weight function $\lambda \in L^\infty(\Omega)$  with $\lambda \geq 0$ a.e., whose regions of discontinuity are resolved by the initial mesh $\TT_0$ (i.e., $g$ is continuous in the interior of every element of $\TT_0$).
Then, we consider the weighted $L^2$-norm
\begin{equation}\label{eq:G:example1}
	G(u)
	= \int_\Omega \lambda(x) u(x)^2 \d{x}
	= \dual{\lambda u}{u}_{H^{-1}\times H^1_0}
	= \norm{\lambda^{1/2}u}{L^2(\Omega)}^2
\end{equation}
as goal functional.
We note that $b(v,w) = \dual{\lambda v}{w}_{H^{-1}\times H^1_0}$ and hence~\eqref{eq:dual:right-hand-side} holds with $g[w] = 2\lambda w$ and $\g[w]  = 0$.
Moreover, we observe that $\KK u = \lambda u \in L^2(\Omega) \hookrightarrow H^{-1}(\Omega)$, where the embedding is compact, so that the goal functional from~\eqref{eq:G:example1} fits in the setting of Theorem~\ref{theorem:main} and Theorem~\ref{theorem:combined}.
We choose
\begin{equation*}
	\lambda(x) =
	\begin{cases}
		1, & x \in U_1,\\
		0, & x \notin U_1,
	\end{cases}
\end{equation*}
with $U_1 = (0.25,0.75)^2$.
This functional is evaluated at the solution of equation~\eqref{eq:strongform} with $f = 2x(x-1) + 2y(y-1)$ and $\f = \boldsymbol{0}$.
The solution of this equation, as well as the value of the goal functional, can be computed analytically to be $u = xy(1-x)(1-y)$ and $G(u) = \int_{U_1} u^2 \d{x} = \frac{41209}{58982400}$, respectively.
The numerical results are visualized in Figure~\ref{fig:convergencel2}.

\begin{figure}
	\centering
	\includegraphics[width=0.9\linewidth]{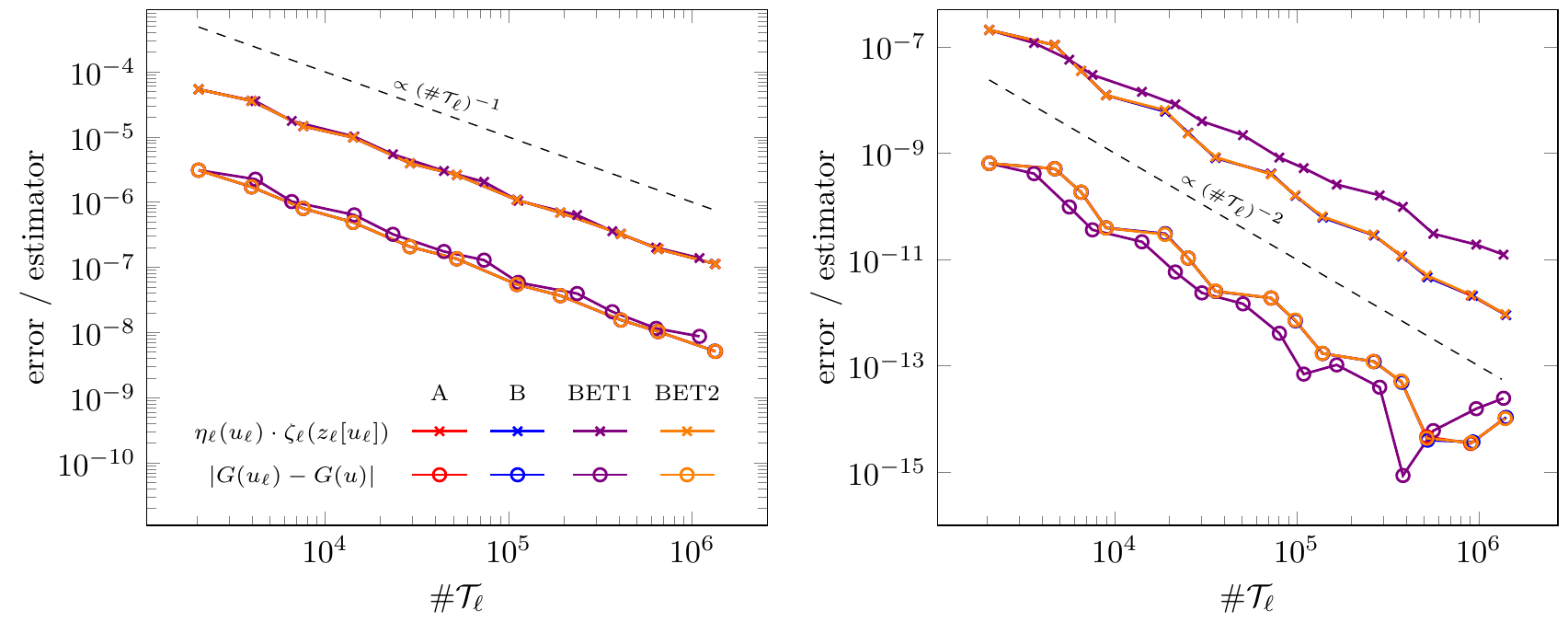}
	\caption{Convergence rates of estimator product and goal error for the problem setting in Section~\ref{subsec:L2} with $p=1$ (left) and $p=2$ (right).
	Note that the lines marked with A, B, and BET2 are almost identical.}
	\label{fig:convergencel2}
\end{figure}

\subsection{Nonlinear convection}\label{subsec:convection}

Suppose that $\boldsymbol{\lambda} \in [L^\infty(\Omega)]^2$ is some vector field, whose regions of discontinuity are resolved by the initial mesh $\TT_0$.
As goal functional we consider the nonlinear convection term
\begin{equation}\label{eq:G:example2}
	G(u)
	= \int_\Omega u(x) \boldsymbol{\lambda}(x) \cdot \nabla u(x) \d{x}
	= \dual{\boldsymbol{\lambda} \cdot \nabla u}{u}_{H^{-1}\times H^1_0}.
\end{equation}
We note that $b(v,w) = \dual{\boldsymbol{\lambda} \cdot \nabla v}{w}_{H^{-1}\times H^1_0}$ and hence~\eqref{eq:dual:right-hand-side} holds with $g[w]  = \boldsymbol{\lambda}\cdot \nabla w$ and $\g[w]  = -w\boldsymbol{\lambda}$.
Moreover, we observe that $\KK u = \boldsymbol{\lambda} \cdot \nabla u \in L^2(\Omega) \hookrightarrow H^{-1}(\Omega)$, where the embedding is compact, so that the goal functional from~\eqref{eq:G:example2} fits in the setting of Theorem~\ref{theorem:main} and Theorem~\ref{theorem:combined}.

We compute the solutions to the primal and the dual problem for $f=0$,
\begin{equation*}
	\f(x) =
	\begin{cases}
		\frac{1}{\sqrt{2}}(-1,1) & \text{if } x \in U_3,\\
		0 & \text{else},
	\end{cases}
	\quad \text{and} \quad
	\boldsymbol{\lambda} =
	\frac{\sigma}{\sqrt{2}} \binom{-1}{1}
	\text{ with }
	\sigma =
	\begin{cases}
		1 & \text{if } x \in U_2, \\
		-1 & \text{else}.
	\end{cases}
\end{equation*}
The sets $U_3 := \set{x \in \Omega}{x_1 - x_2 \geq 0.25}$ and $U_2 := (0.5,1) \!\times\! (0,0.5)$ are shown in Figure~\ref{fig:mesh}.
The numerical results are visualized in Figure~\ref{fig:convergence2}.
Note that the primal problem in this case exhibits a singularity which is not induced by the geometry and thus is not present in the dual problem.

\begin{figure}
	\centering
	\includegraphics[width=0.5\linewidth]{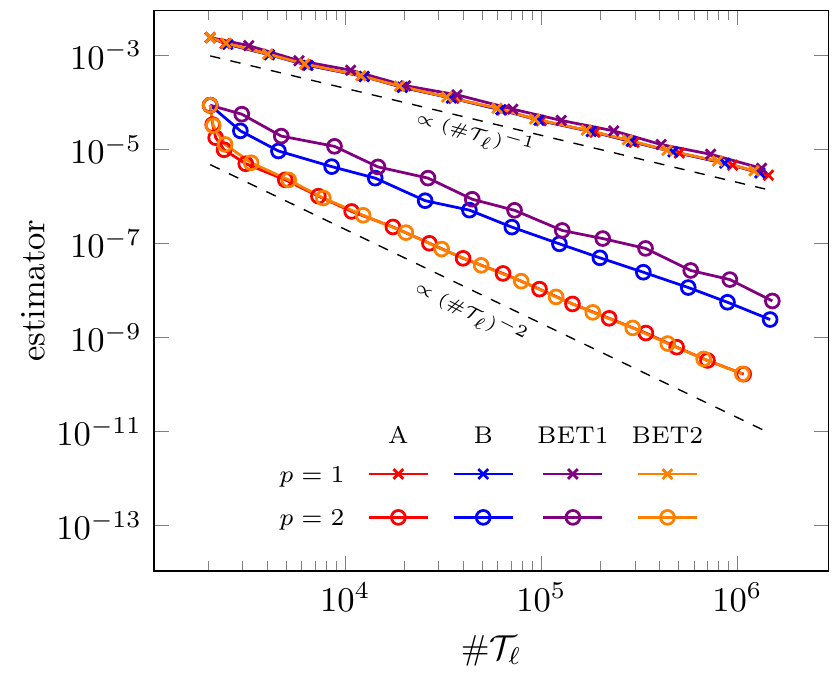}%
	\includegraphics[width=0.5\linewidth]{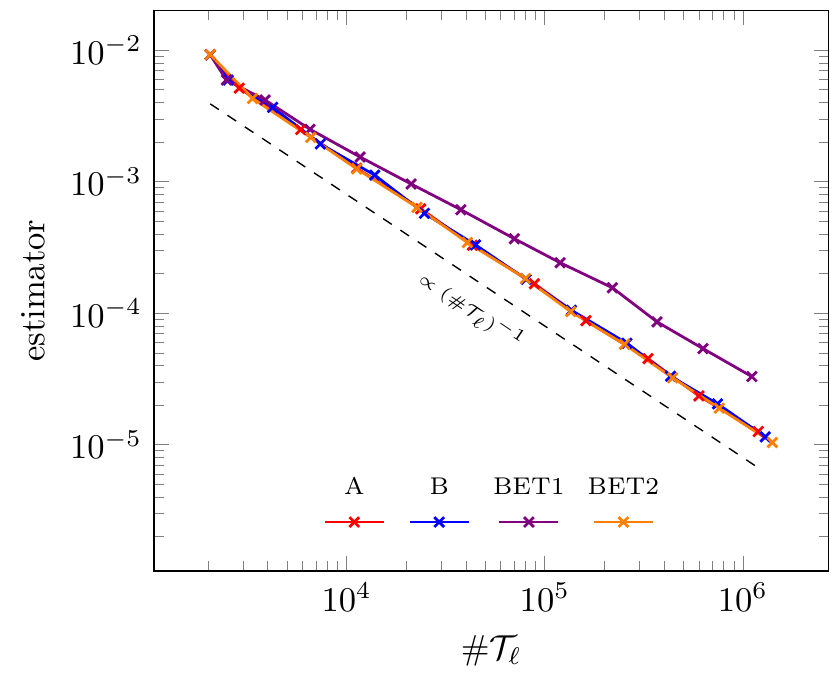}	
	\caption{Convergence rates of estimator product for the problem setting in Section~\ref{subsec:convection} (left) and Section~\ref{subsec:force} (right).}
	\label{fig:convergence2}
\end{figure}

\subsection{Force evaluation}\label{subsec:force}

Let $\varepsilon > 0$ and let $\psi$ be a cut-off function that satisfies
\begin{equation*}
	\psi(x) = 1 \text{ if } x \in U_1
	\quad \text{and} \quad
	\psi(x) = 0 \text{ if } \textrm{dist}(x,U_1) > \varepsilon.
\end{equation*}
For a given direction $\boldsymbol{\chi} \in \R^2$, consider a goal functional of the form
\begin{equation}\label{eq:G:example3}
	G(u)
	:=
	\int_\Omega \nabla \psi \cdot (\nabla u \otimes \nabla u - \tfrac{1}{2} |\nabla u|^2 I) \boldsymbol{\chi} \d{x}.
\end{equation}
This approximates the electrostatic force which is exerted by an electric potential $u$ on a charged body occupying the domain $U_1$ in direction $\boldsymbol{\chi}$ (the part of the integrand in brackets is the so-called Maxwell stress tensor).
We note that
\begin{equation*}
	b(v,w) = \int_\Omega \nabla \psi \cdot \big(\nabla v \otimes \nabla w - \tfrac{1}{2} (\nabla v \cdot \nabla w) I \big) \boldsymbol{\chi} \d{x}
\end{equation*}
and hence~\eqref{eq:dual:right-hand-side} holds with $g[w]  = 0$ and $\g[w]  = (\nabla \psi \cdot \boldsymbol{\chi}) \nabla w - (\nabla \psi \cdot \nabla w) \boldsymbol{\chi} - (\boldsymbol{\chi} \cdot \nabla w) \nabla \psi$.
We stress that the goal functional from~\eqref{eq:G:example3} does not fit in the setting of Theorem~\ref{theorem:main} and Theorem~\ref{theorem:combined}, since the corresponding operator $\KK$ is not compact. Hence, we cannot guarantee optimal rates for our Algorithms~\ref{algorithm}~and~\ref{algorithm:combined}.
However, Proposition~\ref{prop:convergenceA} and Proposition~\ref{prop:convergenceB} still guarantee convergence of our algorithms.

For our experiments, we choose $\boldsymbol{\chi} = \frac{1}{\sqrt{2}} (1,1)^\top$, $f=1$, and $\f = \boldsymbol{0}$.
Furthermore, we choose $\psi$ to be in $\XX_0$ for $p = 1$, i.e., $\psi$ is piecewise linear, and $\varepsilon$ is chosen such that $\psi$ falls off to $0$ exactly within one layer of elements around $U_1$ in $\TT_0$.

The results can be seen in Figure~\ref{fig:convergence2}.

\subsection{Discussion of numerical experiments}

We clearly see from Figures~\ref{fig:convergencel2}--\ref{fig:convergence2} that our Algorithm~\ref{algorithm} and BET2 outperform BET1 and sometimes even Algorithm~\ref{algorithm:combined}.
From Figure~\ref{fig:thetavariation}, where we plot estimator product (and, if available, goal error) for different parameters $\theta = 0.1,0.2, \ldots, 1.0$, we see that this behavior does not depend on the marking parameter $\theta$, generally speaking.
It is striking that the strategy BET1 with $\theta < 1$ fails to drive down the estimator product at the same speed as uniform refinement.
This is likely due to the fact that the linearization error $\enorm{z_\ell[u]-z_\ell[u_\ell]}$ is disregarded; see Remark~\ref{rem:algorithm}.

In Figure~\ref{fig:cumulativecosts}, we plot the cumulative costs
\begin{equation}
\label{eq:cumulative}
	\sum_{\ell \in S[\tau]} \#\TT_\ell,
	\quad\text{with}\quad
	S[\tau] := \set{\ell \in \N}{\textrm{error}_\ell \geq \tau},
\end{equation}
where $\textrm{error}_\ell$ is either the estimator product $\eta_\ell(u_\ell) \zeta_\ell(z_\ell[u_\ell])$, or the goal error $|G(u) - G(u_\ell)|$ in the $\ell$-th step of the adaptive algorithm.
We see that for the setting from Section~\ref{subsec:L2}, where no singularity occurs, optimal costs are achieved by uniform refinement, as is expected.
For the goal error, which is not known in general, the strategy BET1 performs better than our Algorithms~\ref{algorithm}~and~\ref{algorithm:combined}.
However, for the estimator product, which is the relevant quantity in most applications (since the error is unknown), it is inferior.
In the other settings, where there is a singularity, our Algorithms~\ref{algorithm}~and~\ref{algorithm:combined} achieve their minimal cost around the value $0.7$ for the marking parameter $\theta$.

\begin{figure}
	\centering
	\includegraphics[width=0.99\linewidth]{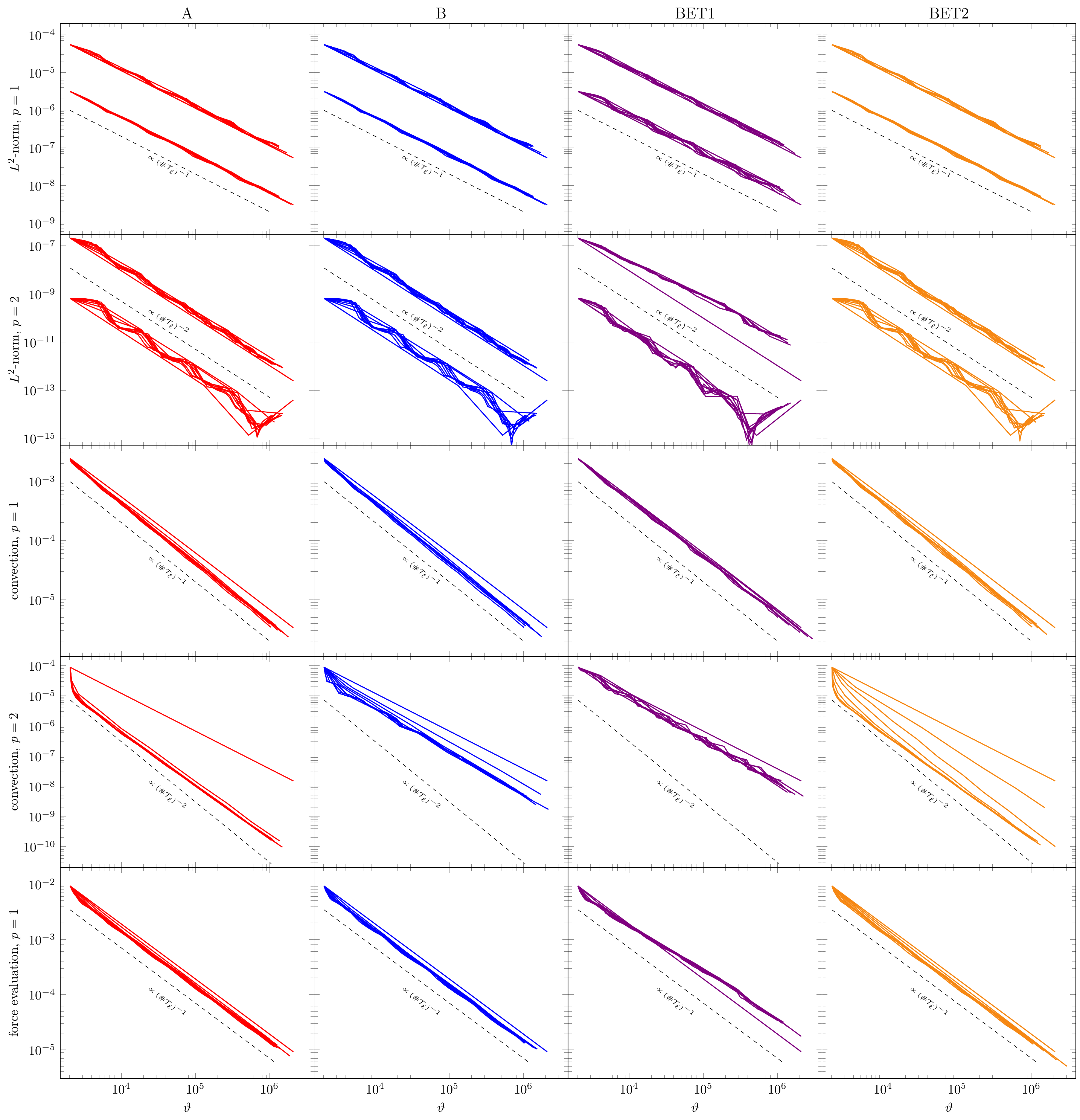}
	\caption{Variation of $\theta$ from $0.1$ (light) to $1.0$ (dark) in steps of $0.1$.
	Left: Setting from Section~\ref{subsec:L2} with $p=2$, where the upper lines represent the estimator product and the lower ones the goal error.
	Middle: Estimator product for the setting from Section~\ref{subsec:convection} with $p=2$.
	Right: Estimator product for the setting from Section~\ref{subsec:force} with $p=1$.}
	\label{fig:thetavariation}
\end{figure}

\begin{figure}
	\centering
	\includegraphics[width=0.8\linewidth]{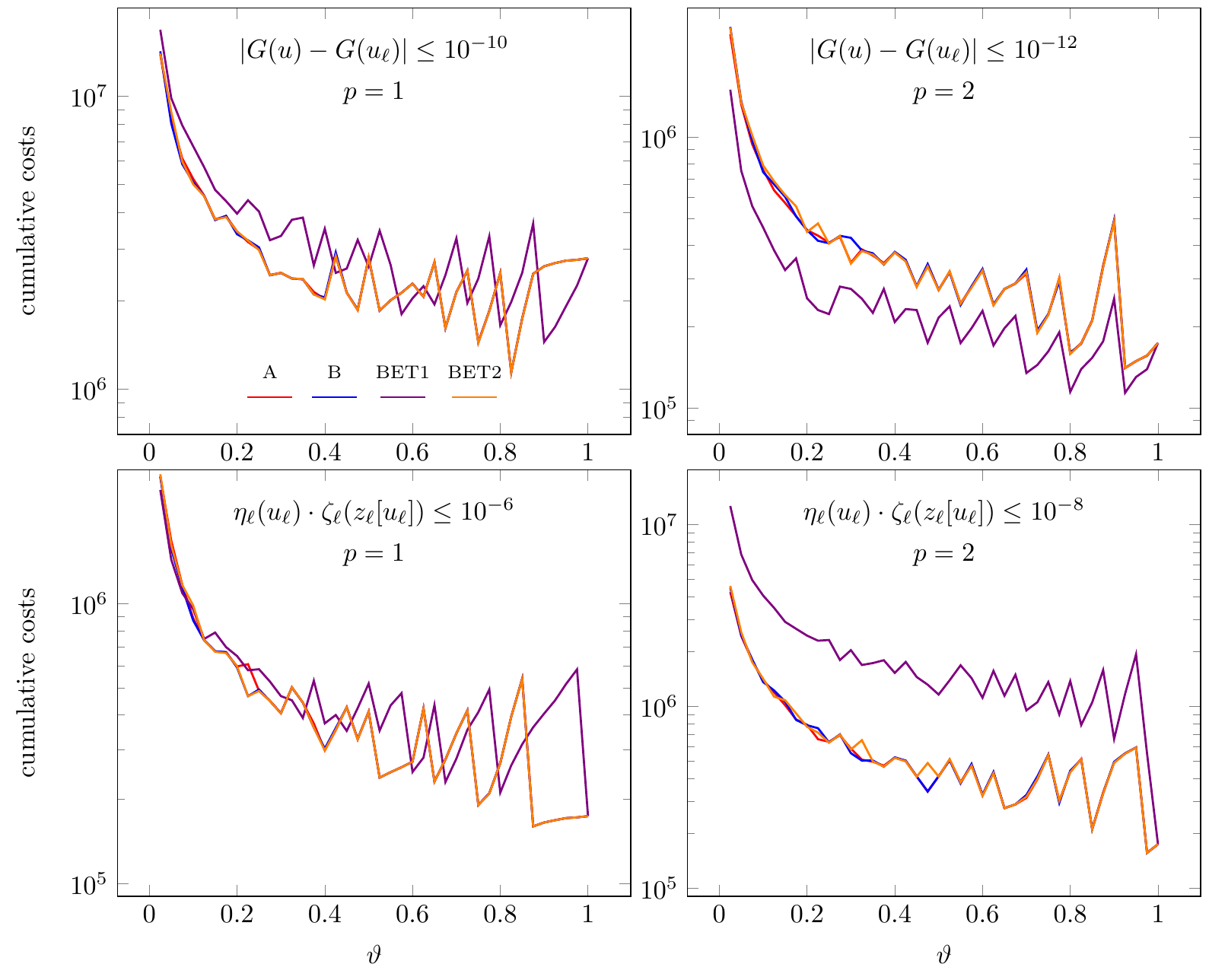}
	\caption{Cumulative costs~\eqref{eq:cumulative} for estimator product and goal error for the setting of Section~\ref{subsec:L2} (top), for the estimator product for the setting of Section~\ref{subsec:convection} (bottom left), and for the estimator product for the setting of Section~\ref{subsec:force} (bottom right).
	The parameters $\theta$ are chosen uniformly in $[0.1,1]$ with stepsize $0.1$.}
	\label{fig:cumulativecosts}
\end{figure}


\section{Auxiliary results}\label{sec:auxiliary}

\subsection{Axioms of adaptivity}

Clearly, $z[w]$ and $z_\coarse[w]$ depend linearly on $w$ (since $\KK$ is linear and hence $b(\cdot,\cdot)$ is bilinear). Moreover, we have the following stability estimates.

\begin{lemma}\label{lemma:stab:z}
	For all $w \in H^1_0(\Omega)$ and all $\TT_\coarse \in \T(\TT_0)$, it holds that 
	\begin{align}\label{eq:stab:z}
	C_1^{-1} \, \enorm{z_\coarse[w]} \le \enorm{z[w]} \le C_2 \, \enorm{w}, 
	\end{align}
	where $C_1 > 0$ depends only on $a(\cdot,\cdot)$, while $C_2 > 0$ depends additionally on the boundedness of $\KK$.
\end{lemma}

\begin{proof}
	The definition of the dual problem shows that
	\begin{align*}
	\enorm{z[w]}^2 \lesssim a(z[w], z[w]) \reff{eq:dual:formulation}= b(z[w],w) + b(w, z[w]) \lesssim \enorm{w} \, \enorm{z[w]}
	\end{align*}
	and hence $\enorm{z[w]} \lesssim \enorm{w}$. Moreover, the stability of the Galerkin method yields that
	$\enorm{z_\coarse[w]} \lesssim \enorm{z[w]}.$
	This concludes the proof.
\end{proof}

Next, we show that the combined estimator for the primal and dual problem satisfies the assumptions~\eqref{assumption:stab}--\eqref{assumption:drel}, where particular emphasis is put on~\eqref{assumption:rel}--\eqref{assumption:drel}.
For the ease of presentation (and by abuse of notation), we use the same constants as for the original properties~\eqref{assumption:stab}--\eqref{assumption:drel}, even though they now depend additionally on the bilinear form $a(\cdot,\cdot)$ and the boundedness of $\KK$.

\begin{proposition}\label{prop:axioms-product-estimator}
	Suppose~\eqref{assumption:stab}--\eqref{assumption:drel} for $\eta_\coarse$ and $\zeta_\coarse$.
	Let $\TT_\coarse \in \T$ and $\TT_\fine \in \T(\TT_\coarse)$.
	Then, \eqref{assumption:stab}--\eqref{assumption:drel} hold also for the combined estimator $\big[ \eta_\coarse(\cdot)^2 + \zeta_\coarse(\cdot)^2 \big]^{1/2}$:
	\begin{enumerate}
		\item[\eqref{assumption:stab}]
		For all $v_\fine, w_\fine \in \XX_\fine$, $v_\coarse, w_\coarse \in \XX_\coarse$, and $\UU_\coarse \subseteq \TT_\fine \cap \TT_\coarse$, it holds that
		\begin{align*}
		&\Big| \big[ \eta_\fine(\UU_\coarse, v_\fine)^2 + \zeta_\fine(\UU_\coarse, w_\fine)^2 \big]^{1/2}
		- \big[ \eta_\coarse(\UU_\coarse, v_\coarse)^2 + \zeta_\coarse(\UU_\coarse, w_\coarse)^2 \big]^{1/2} \Big| \\
		& \hspace{250pt} \le
		\Cstab \, \big[ \enorm{v_\fine - v_\coarse} + \enorm{w_\fine - w_\coarse} \big].
		\end{align*}
		
		\item[\eqref{assumption:red}]
		For all $v_\coarse, w_\coarse \in \XX_\coarse$, it holds that
		\begin{equation*}
		\big[ \eta_\fine(\TT_\fine \backslash \TT_\coarse, v_\coarse)^2 + \zeta_\fine(\TT_\fine \backslash \TT_\coarse, w_\coarse)^2 \big]^{1/2}
		\leq
		\qred \, \big[ \eta_\coarse(\TT_\coarse \backslash \TT_\fine, v_\coarse)^2 + \zeta_\coarse(\TT_\coarse \backslash \TT_\fine, w_\coarse)^2 \big]^{1/2}.
		\end{equation*}
		
		\item[\eqref{assumption:rel}]
		The Galerkin solutions $u_\coarse, z_\coarse[u_\coarse] \in \XX_\coarse$ to~\eqref{eq:discrete:formulation} satisfy that
		\begin{equation*}
		\enorm{u - u_\coarse} + \enorm{z[u] - z_\coarse[u_\coarse]} + \enorm{z[u_\coarse] - z_\coarse[u_\coarse]}
		\leq
		\Crel \, \big[ \eta_\coarse(u_\coarse)^2 + \zeta_\coarse(z_\coarse[u_\coarse])^2 \big]^{1/2}.
		\end{equation*} 
		
		\item[\eqref{assumption:drel}]
		The Galerkin solutions $u_\coarse, z_\coarse[u_\coarse] \in \XX_\coarse$ and $u_\fine, z_\fine[u_\fine] \in \XX_\fine$ to~\eqref{eq:discrete:formulation} satisfy that
		\begin{equation*}
		\enorm{u_\fine - u_\coarse} + \enorm{z_\fine[u_\fine] - z_\coarse[u_\coarse]}
		\leq
		\Cdrel \, \big[ \eta_\coarse(\TT_\coarse \backslash \TT_\fine, u_\coarse)^2 + \zeta_\coarse(\TT_\coarse \backslash \TT_\fine, z_\coarse[u_\coarse])^2 \big]^{1/2}.
		\end{equation*} 
	\end{enumerate}
\end{proposition}

\begin{proof}
	By the triangle inequality and $[a^2 + b^2]^{1/2} \leq a + b$, \eqref{assumption:stab} follows from stability of $\eta_\coarse$ and $\zeta_\coarse$.
	Reduction~\eqref{assumption:red} follows directly from the corresponding properties of $\eta_\coarse$ and $\zeta_\coarse$.
	For \eqref{assumption:rel}, we see with Lemma~\ref{lemma:stab:z} that
	\begin{align*}
	\enorm{z[u] - z_\coarse[u_\coarse]}
	\le \enorm{z[u] - z[u_\coarse]} + \enorm{z[u_\coarse] - z_\coarse[u_\coarse]}
	&\reff{eq:stab:z}\lesssim \enorm{u - u_\coarse} + \enorm{z[u_\coarse] - z_\coarse[u_\coarse]}.
	\end{align*}
	Hence,~\eqref{assumption:rel} follows from reliability of $\eta_\coarse$ and $\zeta_\coarse$. Discrete reliability~\eqref{assumption:drel} follows from the same arguments.
\end{proof}

In the following, we recall some basic results of~\cite{axioms}.

\begin{lemma}[quasi-monotonicity of estimators~{\cite[Lemma~3.6]{axioms}}]\label{lemma:quasi-monotonicity}
	Let $w \in H^1_0(\Omega)$.
	Let $\TT_\coarse \in \T$ and $\TT_\fine \in \T(\TT_\coarse).$
	The properties~\eqref{assumption:stab}--\eqref{assumption:rel} together with the C\'ea lemma guarantee that
	\begin{equation}\label{eq:quasi-monotonicity}
	\eta_\fine(u_\fine)^2 \le \Cmon \, \eta_\coarse(u_\coarse)^2
	\quad \text{as well as} \quad
	\zeta_\fine(z_\fine[w])^2 \le \Cmon \, \zeta_\coarse(z_\coarse[w])^2.
	\end{equation}
	Moreover, the properties~\eqref{assumption:stab}--\eqref{assumption:rel} for the combined error estimator together with the C\'ea lemma show that
	\begin{equation}
		\label{eq:quasi-monotonicity-sum}
		\eta_\fine(u_\fine)^2 + \zeta_\fine(z_\fine[u_\fine])^2
		\le
		\Cmon \, \big[ \eta_\coarse(u_\coarse)^2 + \zeta_\coarse(z_\coarse[u_\coarse])^2 \big].
	\end{equation}
	The constant $\Cmon > 0$ depends only on the properties~\eqref{assumption:stab}--\eqref{assumption:rel} and on the bilinear form $a(\cdot,\cdot)$ and the boundedness of $\KK$.\qed
\end{lemma}

\begin{lemma}[generalized estimator reduction~{\cite[Lemma~4.7]{axioms}}]\label{lemma:generalized_estimator_reduction}
	Let $\TT_\coarse \in \T$ and $\TT_\fine \in \T(\TT_\coarse)$. Let $v_\coarse \in \XX_\coarse$, $v_\fine \in \XX_\fine$, and $\delta > 0$.
	Then,
	\begin{itemize}
		\item $\eta_\fine(v_\fine)^2 \le (1 + \delta) \, \big[ \, \eta_\coarse(v_\coarse)^2 - (1-\qred^2) \, \eta_\coarse(\TT_\coarse \backslash \TT_\fine, v_\coarse)^2 \, \big] + (1 + \delta^{-1}) \, \Cstab^2 \, \enorm{v_\fine - v_\coarse}^2$,
		\item $\zeta_\fine(v_\fine)^2 \le (1 + \delta) \, \big[ \, \zeta_\coarse(v_\coarse)^2 - (1-\qred^2) \, \zeta_\coarse(\TT_\coarse \backslash \TT_\fine, v_\coarse)^2 \, \big] + (1 + \delta^{-1}) \, \Cstab^2 \, \enorm{v_\fine - v_\coarse}^2$.
	\end{itemize}
	If, for instance, $\theta \, \eta_\coarse(u_\coarse)^2 \le \eta_\coarse(\TT_\coarse \backslash \TT_\fine, u_\coarse)^2$ with $0 < \theta \leq 1$, then it follows that
	\begin{align}\label{eq:generalized_estimator_reduction}
	\eta_\fine(u_\fine)^2 \le q \, \eta_\coarse(u_\coarse)^2 + C \, \enorm{u_\fine - u_\coarse}^2.
	\end{align}
	In this case, it holds that $0 < q := (1 + \delta) \, \big[ \, 1 - (1-\qred^2) \, \theta \, \big] < 1$ and $C := (1 + \delta^{-1}) \, \Cstab^2$ with $\delta > 0$ being sufficiently small.\qed
\end{lemma}

\begin{lemma}[optimality of D\"orfler marking~{\cite[Proposition~4.12]{axioms}}]\label{lemma:doerfler}
Suppose stability~\eqref{assumption:stab} and discrete reliability~\eqref{assumption:drel}. For all $0<\theta<\theta_{\rm opt}:=(1+\Cstab^2\Cdrel^2)^{-1}$, there exists some $0<\kappa_{\rm opt}<1$ such that for all $\TT_\ell\in\T$ and all $\TT_\fine \in \T(\TT_\ell)$, it holds that
\begin{align}
 \eta_{\fine}(u_\fine)^2 
 &\le \kappa_{\rm opt} \,\eta_{\ell}(u_\ell)^2
 \quad \Longrightarrow \quad
 \theta\,\eta_{\ell}(u_\ell)^2
 \le \eta_{\ell}(\TT_\ell \backslash \TT_\fine,u_\ell)^2,
 \\
 \begin{split}
 \big[ \eta_{\fine}(u_\fine)^2 + \zeta_{\fine}(z_\fine[u_\fine])^2 \big]
 &\le \kappa_{\rm opt} \, \big[ \eta_\ell^2 + \zeta_{\ell}(z_\ell[u_\ell])^2 \big]
 \\
 &\hspace*{-25mm}\Longrightarrow \quad
 \theta \, \big[ \eta_\ell^2 + \zeta_{\ell}(z_\ell[u_\ell])^2 \big]
 \le \big[ \eta_{\ell}(\TT_\ell \backslash \TT_\fine,u_\ell)^2 + \zeta_{\ell}(\TT_\ell \backslash \TT_\fine,z_\ell[u_\ell])^2 \big].
 \qquad \qed
\end{split}
\end{align}
\end{lemma}

\subsection{Quasi-orthogonality}
\label{section:orthogonality}

To prove linear convergence in the spirit of~\cite{ckns2008}, we imitate the approach from~\cite{bhp2017}. One crucial ingredient are appropriate quasi-orthogonalities. To this end, our proofs exploit the observation of~\cite{ffp2014} that, for any compact operator $\CC : H^1_0(\Omega) \to H^{-1}(\Omega)$, convergence $\norm{u - u_\ell}{H^1(\Omega)} \to 0$ plus Galerkin orthogonality for the nested discrete spaces $\XX_\ell \subseteq \XX_{\ell+1} \subset H^1_0(\Omega)$ for all $\ell \in \N_0$ even yields that $\norm{\CC(u - u_\ell)}{H^{-1}(\Omega)}/\norm{u-u_\ell}{H^1(\Omega)} \to 0$ as $\ell \to \infty$. The latter is also the key argument for the following two lemmas.

\begin{lemma}[quasi-orthogonality for primal problem~{\cite[Lemma~18]{bhp2017}}]
\label{lemma:orth:primal}
Suppose that $\enorm{u - u_\ell} \to 0$ as $\ell \to \infty$. Then,
for all $0 < \eps < 1$, there exists $\ell_0 \in \N$ such that,
$\text{for all } \ell \ge \ell_0 \text{ and all } n \in \N_0,$
\begin{align}\label{eq:orth:primal}
 \enorm{u - u_{\ell+n}}^2 + \enorm{u_{\ell+n} - u_\ell}^2 
 \le \frac{1}{1-\eps} \, \enorm{u - u_\ell}^2. 
 \qquad \qed
\end{align}%
\end{lemma}

The same result holds for the dual problem, if the algorithm ensures convergence $\enorm{z[u] - z_\ell[u]} \to 0$ as $\ell \to \infty$.

\begin{lemma}[quasi-orthogonality for \textit{exact} dual problem~{\cite[Lemma~18]{bhp2017}}]
\label{lemma:orth:dual:exact}
Suppose that $\enorm{z[u] - z_\ell[u]} \to 0$ as $\ell \to \infty$. Then,
for all $0 < \eps < 1$, there exists $\ell_0 \in \N$ such that,
$\text{for all } \ell \ge \ell_0 \text{ and all } n \in \N_0,$
\begin{align}\label{eq:orth:dual:exact}
 \enorm{z[u] - z_{\ell+n}[u]}^2 + \enorm{z_{\ell+n}[u] - z_\ell[u]}^2 
 \le \frac{1}{1-\eps} \, \enorm{z[u] - z_\ell[u]}^2.
\qquad \qed
\end{align}%
\end{lemma}

\begin{lemma}[combined quasi-orthogonality for \textit{inexact} dual problem]\label{lemma:orth:dual}
Suppose that $\enorm{u - u_\ell} + \enorm{z[u] - z_\ell[u_\ell]} \to 0$ as $\ell \to \infty$.
Then, for all $0 < \delta < 1$, there exists $\ell_0 \in \N$ such that,
$\text{for all } \ell \ge \ell_0 \text{ and all } n \in \N_0,$
\begin{align}
\begin{split}
 &\big[\, \enorm{u-u_{\ell+n}}^2 + \enorm{z[u] - z_{\ell+n}[u_{\ell+n}]}^2 \, \big]
 + \big[ \, \enorm{u_{\ell+n} - u_\ell} + \enorm{z_{\ell+n}[u_{\ell+n}] - z_\ell[u_\ell]}^2 \, \big]
\\& \qquad
 \le \frac{1}{1-\delta} \, \big[ \, \enorm{u - u_\ell}^2 + \enorm{z[u] - z_\ell[u_\ell]}^2 \, \big].
\end{split}
\end{align}
\end{lemma}

\begin{proof}
According to Lemma~\ref{lemma:stab:z}, it holds that
\begin{align*}
 \enorm{z[u] - z_\ell[u]}
 \le \enorm{z[u] - z_\ell[u_\ell]} + \enorm{z_\ell[u] - z_\ell[u_\ell]}
 \reff{eq:stab:z}\lesssim \enorm{z[u] - z_\ell[u_\ell]} + \enorm{u - u_\ell}
 \xrightarrow{\ell \to \infty} 0.
\end{align*}
Hence, we may exploit the conclusions of Lemma~\ref{lemma:orth:primal} and Lemma~\ref{lemma:orth:dual:exact}.
For arbitrary $\alpha > 0$, the Young inequality guarantees that 
\begin{align*}
 \enorm{z[u] - z_{\ell+n}[u_{\ell+n}]}^2
 &\le (1 + \alpha) \, \enorm{z[u] - z_{\ell+n}[u]}^2
 + (1 + \alpha^{-1}) \, \enorm{z_{\ell+n}[u] - z_{\ell+n}[u_{\ell+n}]}^2,
 \\
 \enorm{z_{\ell+n}[u_{\ell+n}] - z_\ell[u_\ell]}^2
 &\le (1+\alpha) \, \enorm{z_{\ell+n}[u] - z_{\ell}[u]}^2
 + (1+\alpha^{-1})^2 \, \enorm{z_{\ell}[u] - z_{\ell}[u_\ell]}^2 
 \\ &\qquad
 + (1+\alpha)(1+\alpha^{-1}) \, \enorm{z_{\ell+n}[u] - z_{\ell+n}[u_{\ell+n}]}^2,
 \\
 \enorm{z[u] - z_{\ell}[u]}^2
 &\le (1 + \alpha) \, \enorm{z[u] - z_{\ell}[u_{\ell}]}^2
 + (1 + \alpha^{-1}) \, \enorm{z_{\ell}[u] - z_{\ell}[u_{\ell}]}^2.
\end{align*}
Together with Lemma~\ref{lemma:orth:dual:exact}, this leads to
\begin{align}\label{eq1:orth:dual}
\begin{split}
 &\enorm{z[u] - z_{\ell+n}[u_{\ell+n}]}^2
 + \enorm{z_{\ell+n}[u_{\ell+n}] - z_\ell[u_\ell]}^2
 \\ & \quad
 \refp{eq:orth:dual:exact}\le (1+\alpha) \, \big[ \, \enorm{z[u] - z_{\ell+n}[u]}^2 + \enorm{z_{\ell+n}[u] - z_{\ell}[u]}^2 \, \big]
 \\& \qquad
 + (2+\alpha)(1+\alpha^{-1}) \, \enorm{z_{\ell+n}[u] - z_{\ell+n}[u_{\ell+n}]}^2
 + (1+\alpha^{-1})^2 \, \enorm{z_{\ell}[u] - z_{\ell}[u_\ell]}^2
 \\ & \quad
 \reff{eq:orth:dual:exact}\le \frac{1+\alpha}{1-\eps} \, \enorm{z[u] - z_\ell[u]}^2
 + (1+\alpha^{-1})^2 \, \enorm{z_{\ell}[u] - z_{\ell}[u_\ell]}^2
 \\& \qquad
 + (2+\alpha)(1+\alpha^{-1}) \, \enorm{z_{\ell+n}[u] - z_{\ell+n}[u_{\ell+n}]}^2
 \\ & \quad
 \refp{eq:orth:dual:exact}\le \frac{(1+\alpha)^2}{1-\eps} \, \enorm{z[u] - z_{\ell}[u_{\ell}]}^2
 + \Big[ (1+\alpha^{-1})^2 + \frac{(1+\alpha^{-1})(1+\alpha)}{1-\eps} \Big] \, \enorm{z_{\ell}[u] - z_{\ell}[u_\ell]}^2
 \hspace*{-5mm}
 \\& \qquad
 + (2+\alpha)(1+\alpha^{-1}) \, \enorm{z_{\ell+n}[u] - z_{\ell+n}[u_{\ell+n}]}^2
\end{split}
\end{align}
for all $0 < \eps < 1$ and all $\ell \ge \ell_0$, where $\ell_0 \in \N_0$ depends only on $\eps$.
With the (compact) adjoint $\KK' : H^1_0(\Omega) \to H^{-1}(\Omega)$ of $\KK$, we note that
\begin{align*}
 &\enorm{z_{\ell}[u] - z_{\ell}[u_\ell]}^2
 = \enorm{z_{\ell}[u - u_\ell]}^2
 \lesssim a(z_{\ell}[u - u_\ell], z_{\ell}[u - u_\ell])
 \\& \quad
 \refp{eq:stab:z}= b(z_{\ell}[u - u_\ell], u - u_\ell) + b(u - u_\ell, z_{\ell}[u - u_\ell])
 \\& \quad
 \refp{eq:stab:z}= \dual{\KK(z_{\ell}[u - u_\ell])}{u - u_\ell}_{H^{-1}\times H^1_0}
 + \dual{\KK(u - u_\ell)}{z_{\ell}[u - u_\ell]}_{H^{-1}\times H^1_0}
 \\& \quad
 \refp{eq:stab:z}= \dual{\KK'(u - u_\ell)}{z_{\ell}[u - u_\ell]}_{H^{-1}\times H^1_0}
 + \dual{\KK(u - u_\ell)}{z_{\ell}[u - u_\ell]}_{H^{-1}\times H^1_0}.
 \\& \quad
 \reff{eq:stab:z}\lesssim \big[ \, \norm{\KK'(u - u_\ell)}{H^{-1}(\Omega)} + \norm{\KK(u - u_\ell)}{H^{-1}(\Omega)} \, \big]
 \, \enorm{u - u_\ell}.
\end{align*}
Since $\KK$ and $\KK'$ are compact operators (according to the Schauder theorem), it follows from~\cite[Lemma~3.5]{ffp2014} (see also~\cite[Lemma~17]{bhp2017}) that
\begin{align*}
 \big[ \, \norm{\KK'(u - u_\ell)}{H^{-1}(\Omega)} + \norm{\KK(u - u_\ell)}{H^{-1}(\Omega)} \, \big] 
 \le \widetilde{\kappa}_\ell \, \enorm{u - u_\ell}
 \quad \text{with} \quad
 0 \le \widetilde{\kappa}_\ell \xrightarrow{\ell \to \infty} 0.
\end{align*}
Combining the two last estimates, we see that
\begin{align}\label{eq2:orth:dual}
 \enorm{z_{\ell}[u] - z_{\ell}[u_\ell]}^2 \le \kappa_\ell \, \enorm{u - u_\ell}^2
 \quad \text{for all } \ell \in \N_0, \text{ where }
 0 \le \kappa_\ell \xrightarrow{\ell \to \infty} 0.
\end{align}
Plugging~\eqref{eq2:orth:dual} into~\eqref{eq1:orth:dual}, we thus have shown that
\begin{align*}
 &\enorm{z[u] - z_{\ell+n}[u_{\ell+n}]}^2
 + \enorm{z_{\ell+n}[u_{\ell+n}] - z_\ell[u_\ell]}^2
 \\ & \quad
 \le \frac{(1+\alpha)^2}{1-\eps} \, \enorm{z[u] - z_{\ell}[u_{\ell}]}^2
 + \Big[ (1+\alpha^{-1})^2 + \frac{(1+\alpha^{-1})(1+\alpha)}{1-\eps} \Big] \, \kappa_\ell \, \enorm{u - u_\ell}^2
 \\ & \qquad
 + (2+\alpha)(1+\alpha^{-1}) \, \kappa_{\ell+n} \, \enorm{u - u_{\ell+n}}^2
\end{align*}
for all $0 < \eps < 1$, all $\alpha > 0$, and all $\ell \ge \ell_0$, where $\ell_0 \in \N_0$ depends only on $\eps$.
We combine this estimate with that of Lemma~\ref{lemma:orth:primal}. This leads to 
\begin{align*}
 & \big[\, \enorm{u-u_{\ell+n}}^2 + \enorm{z[u] - z_{\ell+n}[u_{\ell+n}]}^2 \, \big]
 + \big[ \, \enorm{u_{\ell+n} - u_\ell} + \enorm{z_{\ell+n}[u_{\ell+n}] - z_\ell[u_\ell]}^2 \, \big]
 \\& \quad
 \le  C(\alpha,\eps,\ell) \, \big[ \, \enorm{u-u_\ell}^2 + \enorm{z[u] - z_{\ell}[u_{\ell}]}^2 \, \big] 
 + (2+\alpha)(1+\alpha^{-1}) \, \kappa_{\ell+n} \, \enorm{u - u_{\ell+n}}^2,
\end{align*}
where
\begin{equation*}
	C(\alpha,\eps,\ell)
	:=
	\max\bigg\{\frac{(1+\alpha)^2}{1-\eps} \,,\, \frac{1}{1-\eps} + \Big[ (1+\alpha^{-1})^2 + \frac{(1+\alpha^{-1})(1+\alpha)}{1-\eps} \Big] \, \kappa_\ell \bigg\}
\end{equation*}
for all $0 < \eps < 1$, all $\alpha > 0$, and all $\ell \ge \ell_0$, where $\ell_0 \in \N_0$ depends only on $\eps$.
For arbitrary $0 < \alpha, \beta, \eps < 1$, there exists $\ell_0' \in \N_0$ such that for all $\ell \ge \ell_0'$, it holds that 
\begin{align*}
(2+\alpha)(1+\alpha^{-1}) \kappa_\ell \le \beta
\end{align*}
as well as 
\begin{align*}
\frac{1}{1-\eps} + \Big[ (1+\alpha^{-1})^2 + \frac{(1+\alpha^{-1})(1+\alpha)}{1-\eps} \Big] \, \kappa_\ell
\le \frac{(1+\alpha)^2}{1-\eps}.
\end{align*}
Hence, we are led to
\begin{align}\label{eq3:orth:dual}
\begin{split}
 &\big[\, \enorm{u-u_{\ell+n}}^2 + \enorm{z[u] - z_{\ell+n}[u_{\ell+n}]}^2 \, \big]
 + \big[ \, \enorm{u_{\ell+n} - u_\ell} + \enorm{z_{\ell+n}[u_{\ell+n}] - z_\ell[u_\ell]}^2 \, \big]
 \\& \quad
 \le \frac{(1+\alpha)^2}{(1-\eps)(1-\beta)} \, \big[ \, \enorm{u-u_\ell}^2 + \enorm{z[u] - z_{\ell}[u_{\ell}]}^2 \, \big].
\end{split}
\end{align}
Given $0 < \delta < 1$, we first fix $\alpha > 0$ such that $(1+\alpha)^2 < \frac{1}{1-\delta}$. Then, we choose $0 < \eps, \beta < 1$ such that $\frac{(1+\alpha)^2}{(1-\eps)(1-\beta)} \le \frac{1}{1-\delta}$. The choices of $\eps$ and $\beta$ also provide some index $\ell_0 \in \N_0$ such that estimate~\eqref{eq3:orth:dual} holds for all $\ell \ge \ell_0$. This concludes the proof.
\end{proof}


\section{Proof of plain convergence of Algorithm~\lowercase{\ref{algorithm}}~and~\lowercase{\ref{algorithm:combined}}}

\subsection{Algorithm A}
First, we prove the upper bound for the goal error.

\begin{proof}[\bfseries Proof of Proposition~\ref{prop:convergenceA}\rm\bf(i)]
	It holds that
	\begin{align*}
	\big| G(u) - G(u_\coarse) \big|
	&\reff{eq:ansart:aposteriori}{\lesssim}
	\enorm{u - u_\coarse} \, \big[ \, \enorm{z[u_\coarse] - z_\coarse[u_\coarse]} + \enorm{u - u_\coarse} \, \big] \\
	&\reff{assumption:rel}{\lesssim}
	\eta_\coarse(u_\coarse) \, \big[ \, \zeta_\coarse(z_\coarse[u_\coarse]) + \eta_\coarse(u_\coarse) \, \big].
	\end{align*}
	The hidden constants depend only on the boundedness of $a(\cdot,\cdot)$ and $\KK$ and on the constant $\Crel$ from~\ref{assumption:rel}.
	According to the Young inequality, this concludes the proof.
\end{proof}

Since Algorithm~\ref{algorithm} linearizes the dual problem around the known discrete solution (i.e., it employs $z_\ell[u_\ell]$ instead of the non-computable $z_\ell[u]$), a first important observation is that Algorithm~\ref{algorithm} ensures convergence for the primal solution. In particular, the following proposition allows to apply the quasi-orthogonalities from Section~\ref{section:orthogonality}.

\begin{proposition}[plain convergence of errors and error estimators]
\label{prop:plain_convergence}
Suppose~\eqref{assumption:stab}--\eqref{assumption:rel}.
Then, for any choice of marking parameters $0 < \theta \le 1$ and $\Cmark \ge 1$, Algorithm~\ref{algorithm} and Algorithm~\ref{algorithm:combined} guarantee that
\begin{itemize}
\item $\enorm{u-u_\ell} + \eta_\ell(u_\ell) \to 0$
 if $\#\set{k \in \N_0}{\MM_k \text{ satisfies~\eqref{eq:doerfler:u}}} = \infty$,
\item $\enorm{u - u_\ell} + \eta_\ell(u_\ell) + \enorm{z[u] - z_\ell[u_\ell]} + \enorm{z[u] - z_\ell[u]} + \zeta_\ell(z_\ell[u_\ell]) \to 0$
 if $\#\set{k \in \N_0}{\MM_k \text{ satisfies~\eqref{eq:doerfler:uz}}} = \infty$,
\end{itemize}
as $\ell \to \infty$. 
Moreover, at least one of these two cases is met.
\end{proposition}

\begin{proof}
Since the discrete spaces are nested, it follows from the C\'ea lemma that there exists $u_\infty \in H^1_0(\Omega)$ such that
\begin{align}\label{eq:apriori:u}
 \enorm{u_\infty - u_\ell} \xrightarrow{\ell \to \infty} 0;
\end{align}
see, e.g.,~\cite{msv2008,afp2012} or even the early work~\cite{bv1984}. More precisely, $u_\infty$ is the Galerkin approximation of $u$ with respect to the ``discrete limit space'' $\XX_\infty := \overline{\bigcup_{\ell=0}^\infty \XX_\ell}$, where the closure is taken in $H^1_0(\Omega)$. Analogously, there exists $z_\infty \in H^1_0(\Omega)$ such that
\begin{align*}
 \enorm{z_\infty - z_\ell[u_\infty]}  \xrightarrow{\ell \to \infty} 0.
\end{align*}
Together with~\eqref{eq:apriori:u} and Lemma~\ref{lemma:stab:z}, this also proves that
\begin{align*}
 \enorm{z_\infty - z_\ell[u_\ell]}
 &\le \enorm{z_\infty - z_\ell[u_\infty]} + \enorm{z_\ell[u_\infty] - z_\ell[u_\ell]}
 \\&
 \lesssim \enorm{z_\infty - z_\ell[u_\infty]} + \enorm{u_\infty - u_\ell}
 \xrightarrow{\ell \to \infty}0.
\end{align*}

In the following, we aim to show that, in particular, $u = u_\infty$. To this end, the proof considers two cases:
\begin{itemize}
\item[{$\boldsymbol{\langle1\rangle}$}] There exists a subsequence $(\TT_{\ell_k})_{k \in \N_0}$ such that $\MM_{\ell_k}$ satisfies \eqref{eq:doerfler:u} for all $k \in \N_0$,
\item[{$\boldsymbol{\langle2\rangle}$}] There exists a subsequence $(\TT_{\ell_k})_{k \in \N_0}$ such that $\MM_{\ell_k}$ satisfies \eqref{eq:doerfler:uz} for all $k \in \N_0$.
\end{itemize}
Clearly, (at least) one of these subsequences is well-defined (i.e., there are infinitely many steps of the respective marking).

{\bf Case {$\boldsymbol{\langle1\rangle}$}.}
According to Lemma~\ref{lemma:generalized_estimator_reduction}, there exists $0 < q < 1$ and $C > 0$ such that
\begin{align*}
 \eta_{\ell_{k+1}}(u_{\ell_{k+1}})^2 
 \reff{eq:generalized_estimator_reduction}\le q \, \eta_{\ell_k}(u_{\ell_k})^2 + C \, \enorm{u_{\ell_{k+1}} - u_{\ell_k}}^2
 \quad \text{for all } k \in \N_0.
\end{align*}
With~\eqref{eq:apriori:u}, the last estimate proves that the estimator subsequence is contractive up to some zero sequence. Therefore, it follows from basic calculus and reliability~\eqref{assumption:rel} that
\begin{align*}
 \enorm{u - u_{\ell_k}} \reff{assumption:rel}\lesssim \eta_{\ell_{k}}(u_{\ell_{k}}) \xrightarrow{k \to \infty} 0;
\end{align*}
see, e.g.,~\cite[Lemma~2.3]{afp2012}. In particular, this proves that $u = u_\infty$ and hence $\enorm{u - u_\ell} \to 0$ as $\ell \to \infty$. Moreover, according to quasi-monotonicity (Lemma~\ref{lemma:quasi-monotonicity}), the convergence of the subsequence $\eta_{\ell_{k}}(u_{\ell_{k}}) \to 0$ even yields that $\eta_\ell(u_\ell) \to 0$ as $\ell \to \infty$.

{\bf Case {$\boldsymbol{\langle2\rangle}$}.} We repeat the arguments from case~$\langle1\rangle$. Instead of $\eta_\coarse(u_\coarse)^2$, we consider the combined estimator $\eta_\coarse(u_\coarse)^2 + \zeta_\coarse(z_\coarse[u_\coarse])^2$. For all $k \in \N_0$, this leads to
\begin{align*}
 \eta_{\ell_{k+1}}(u_{\ell_{k+1}})^2 + \zeta_{\ell_{k+1}}(z_{\ell_{k+1}}[u_{\ell_{k+1}}])^2
 &\le q \, \big[ \eta_{\ell_k}(u_{\ell_k})^2 + \zeta_{\ell_{k}}(z_{\ell_{k}}[u_{\ell_{k}}])^2 \big] 
 \\& \qquad
 + C \, \big[ \, \enorm{u_{\ell_{k+1}} - u_{\ell_k}}^2 + \enorm{z_{\ell_{k+1}}[u_{\ell_{k+1}}] - z_{\ell_k}[u_{\ell_k}]}^2\, \big].
\end{align*}
As before, basic calculus reveals that
\begin{align*}
 \enorm{u - u_{\ell_k}}^2 + \enorm{z[u] - z_{\ell_k}[u_{\ell_k}]}^2
 \reff{assumption:rel}\lesssim \eta_{\ell_k}(u_{\ell_k})^2 + \zeta_{\ell_k}(z_{\ell_k}[u_{\ell_k}])^2
 \xrightarrow{k \to \infty} 0.
\end{align*}
In this case, we thus see that $u = u_\infty$ and $z[u] = z_\infty$ as well as estimator convergence $\eta_\ell(u_\ell) + \zeta_\ell(z_\ell[u_\ell]) \to 0$ as $\ell \to \infty$ (following now from Lemma~\ref{lemma:quasi-monotonicity}). In any case, this concludes the proof.
\end{proof}

\begin{proof}[\bfseries Proof of Proposition~\ref{prop:convergenceA}\rm\bf(ii)]
Recall from Proposition~\ref{prop:convergenceA}(i) that
\begin{equation*}
 \big| G(u) - G(u_\ell) \big|
 \reff{eq:goal-upper-bound}\lesssim \eta_\ell(u_\ell) \, \big[ \, \eta_\ell(u_\ell)^2 + \zeta_\ell(z_\ell[u_\ell])^2 \, \big]^{1/2}.
\end{equation*}
Suppose $\#\set{k \in \N_0}{\MM_{\ell} ~\text{satisfies}~ \eqref{eq:doerfler:u}} = \infty$.
According to Proposition~\ref{prop:plain_convergence}, it holds that $\eta_\ell(u_\ell) \to 0$ as $\ell \to \infty$.
According to Lemma~\ref{lemma:quasi-monotonicity}, it holds that $\eta_\ell(u_\ell)^2 + \zeta_\ell(z_\ell[u_\ell])^2 \lesssim \eta_0(u_0)^2 + \zeta_0(z_0[u_0])^2 < \infty$.
Thus, the right-hand side of \eqref{eq:plain-convergenceA} vanishes.
In the case $\#\set{k \in \N_0}{\MM_{\ell} ~\text{satisfies}~ \eqref{eq:doerfler:uz}} = \infty$ one can argue analogously.
This concludes the proof.
\end{proof}

\subsection{Algorithm B}
First, we prove the upper bound.

\begin{proof}[\bfseries Proof of Proposition~\ref{prop:convergenceB}\rm\bf(i)]
	From Proposition~\ref{prop:convergenceA}(i) it follows that
	\begin{align*}
	\big| G(u) - G(u_\coarse) \big|
	&\stackrel{\mathclap{\eqref{eq:goal-upper-bound}}}{\leq}
	\Crel^\prime \, \eta_\coarse(u_\coarse) \, \big[ \,
	\eta_\coarse(u_\coarse)^2 + \zeta_\coarse(z_\coarse[u_\coarse])^2
	\, \big]^{1/2} \\
	&\leq
	\Crel^\prime \, \big[ \,
	\eta_\coarse(u_\coarse)^2 + \zeta_\coarse(z_\coarse[u_\coarse])^2
	\, \big].
	\end{align*}
	This proves the claim.
\end{proof}

\begin{proof}[\bfseries Proof of Proposition~\ref{prop:convergenceB}\rm\bf(ii)]
	Recall from Proposition~\ref{prop:convergenceB}(i) that
	\begin{equation*}
		\big| G(u) - G(u_\ell) \big|
		\reff{eq:goal-upper-bound}\lesssim
		\big[ \, \eta_\ell(u_\ell)^2 + \zeta_\ell(z_\ell[u_\ell])^2 \, \big].
	\end{equation*}
	Note that the marking step~\eqref{eq:doerfler:dual} of Algorithm~\ref{algorithm:combined} implies, in particular, that $\#\set{k \in \N_0}{\MM_{\ell} ~\text{satisfies}~ \eqref{eq:doerfler:uz}} = \infty$.
	According to Proposition~\ref{prop:plain_convergence}, it holds that $\eta_\ell(u_\ell) + \zeta_\ell(z_\ell[u_\ell])\to 0$ as $\ell \to \infty$.
	Thus, the right-hand side of \eqref{eq:plain-convergenceB} vanishes.
	This concludes the proof.
\end{proof}


\section{Proof of Theorem~\lowercase{\ref{theorem:main}}}

\subsection{Linear convergence}

Based on estimator reduction (Lemma~\ref{lemma:generalized_estimator_reduction}) and quasi-orthog\-o\-nal\-i\-ty (Lemma~\ref{lemma:orth:primal}, Lemma~\ref{lemma:orth:dual}), we are in the position to address linear convergence.

\begin{proposition}[generalized contraction]\label{prop:linear-convergence}
	Suppose~\eqref{assumption:stab}--\eqref{assumption:rel}.
	Then, there exist constants $\gamma > 0$ and $0 < \qctr < 1$ such that the quasi-errors
	\begin{align}\label{def:Delta}
	\Delta_\coarse^u := \enorm{u - u_\coarse}^2 + \gamma \, \eta_\coarse(u_\coarse)^2
	\quad \text{and} \quad
	\Delta_\coarse^z := \enorm{z[u] - z_\coarse[u_\coarse]}^2 + \gamma \, \zeta_\coarse(z_\coarse[u_\coarse])^2,
	\end{align}
	defined for all $\TT_\coarse \in \T$, satisfy the following contraction properties: There exists
	an index $\ell_0 \ge 0$ such that, for all $\ell\in \N_0$ with $\ell \ge \ell_0$ and all $n \in \N$, it holds that
	\begin{itemize}
		\item $\Delta_{\ell + n}^u \le \qctr \, \Delta_\ell^u$ provided that $\MM_\ell$ satisfies~\eqref{eq:doerfler:u};
		\item $\big[ \, \Delta_{\ell + n}^u + \Delta_{\ell + n}^z \, \big] \le \qctr \, \big[ \, \Delta_\ell^u + \Delta_\ell^z \, \big]$ provided that $\MM_\ell$ satisfies~\eqref{eq:doerfler:uz};
	\end{itemize}
	The constants $\gamma$ and $\qctr$ depend only on $\theta$, $\qred$, $\Cstab$, and $\Crel$, and the index $\ell_0$, which depends on $\gamma$ and $\qctr$, is essentially provided by Lemma~\ref{lemma:orth:primal} and Lemma~\ref{lemma:orth:dual}.
\end{proposition}

\begin{proof}
	Let $0 < \eps, \delta, \gamma < 1$ be free parameters, which will be fixed later.
	
	{\bf Step~1.}
	Consider the case that $\#\set{k \in \N_0}{\MM_k \text{ satisfies~\eqref{eq:doerfler:u}}} = \infty$ and that $\MM_\ell$ satisfies~\eqref{eq:doerfler:u}. 
	From the generalized estimator reduction (Lemma~\ref{lemma:generalized_estimator_reduction}), we get that
	\begin{align*}
	\eta_{\ell+n}(u_{\ell+n})^2 
	\reff{lemma:generalized_estimator_reduction}\le q \, \eta_\ell(u_\ell)^2 + C \, \enorm{u_{\ell+n} - u_\ell}^2,
	\end{align*}
	where $0 < q < 1$ depends only on $\theta$ and $\qred$, while $C > 0$ depends additionally on $\Cstab$.
	Together with the quasi-orthogonality (Lemma~\ref{lemma:orth:primal}), we see that
	\begin{align*}
	\Delta_{\ell+n}^u 
	&= \enorm{u - u_{\ell+n}}^2 + \gamma \, \eta_\ell(u_{\ell+n})^2 
	\\&
	\le \frac{1}{1-\eps} \, \enorm{u - u_\ell}^2 + \gamma q \, \eta_\ell(u_\ell)^2 + (\gamma C - 1) \, \enorm{u_{\ell+n} - u_\ell}^2.
	\end{align*}
	The choice of $\gamma$ must enforce that $\gamma C\le 1$. Together with reliability, we are then led to
	\begin{align*}
	\Delta_{\ell+n}^u 
	\reff{assumption:rel}\le \Big[ \, \frac{1}{1-\eps} - \gamma \delta \Big] \, \, \enorm{u - u_\ell}^2 + \gamma \, \big[ \, q + \delta \Crel^2 \, \big] \, \eta_\ell(u_\ell)^2.
	\end{align*}
	The choice of $\delta > 0$ must guarantee that $q + \delta \Crel^2 < 1$. Finally, the choice of $\eps > 0$ must guarantee that $(1-\eps)^{-1} - \gamma\delta < 1$. Then, we see that
	\begin{align*}
	\Delta_{\ell+n}^u \le \qctr \, \Delta_\ell,
	\quad \text{where} \quad
	\qctr := \max\{ (1-\eps)^{-1} - \gamma\delta \,,\, q + \delta\Crel^2 \} < 1.
	\end{align*}
	
	{\bf Step~2.}
	Consider the case that $\#\set{k \in \N_0}{\MM_k \text{ satisfies~\eqref{eq:doerfler:uz}}} = \infty$ and that $\MM_\ell$ satisfies~\eqref{eq:doerfler:uz}. 
	The same arguments apply (now based on the combined quasi-orthogonality from Lemma~\ref{lemma:orth:dual}). 
	
	{\bf Step~3.}
	If $\ell_0^u := \#\set{k \in \N_0}{\MM_k \text{ satisfies~\eqref{eq:doerfler:u}}} < \infty$, choose $\ell_0 > \ell_0^u$ as well as the free parameters according to Step~2.
	
	{\bf Step~4.}
	If $\ell_0^{uz} := \#\set{k \in \N_0}{\MM_k \text{ satisfies~\eqref{eq:doerfler:uz}}} < \infty$, choose $\ell_0 > \ell_0^{uz}$ as well as the free parameters according to Step~1.
	
	{\bf Step~5.}
	Finally, note that Step~3 and Step~4 are exclusive, since $\N_0 = \set{k \in \N_0}{\MM_k \text{ satisfies~\eqref{eq:doerfler:u}}}  \cup \set{k \in \N_0}{\MM_k \text{ satisfies~\eqref{eq:doerfler:uz}}} $. This concludes the proof.
\end{proof}

\begin{proof}[\bfseries Proof of Theorem~\ref{theorem:main}\rm\bf(i)]
	Recall the quasi-errors from~\eqref{def:Delta}. We first prove that $\Delta_\coarse^u \simeq \eta_\coarse(u_\coarse)^2$ as well as
	$\big[ \, \Delta_\coarse^u + \Delta_\coarse^z \, \big] \simeq \big[ \, \eta_\coarse(u_\coarse)^2 + \zeta_\coarse(z_\coarse[u_\coarse])^2 \, \big]$. To see this, note that
	\begin{align*}
	\gamma \, \eta_\coarse(u_\coarse)^2 
	\le \Delta_\coarse^u
	&\reff{assumption:rel}\le (\Crel^2 + \gamma) \, \eta_\coarse(u_\coarse)^2
	\intertext{as well as}
	\gamma \, \big[ \, \eta_\coarse(u_\coarse)^2 + \zeta_\coarse(z_\coarse[u_\coarse])^2 \, \big]
	\le \Delta_\coarse^u + \Delta_\coarse^{z}
	&\reff{assumption:rel}\le
	(2 \, \Crel^2 + \gamma) \, \big[ \, \eta_\coarse(u_\coarse)^2 + \zeta_\coarse(z_\coarse[u_\coarse])^2 \, \big].
	\end{align*}
	Let $\ell \geq \ell_0$. In $n$ steps, the adaptive algorithm satisfies $k$ times $\overline\MM_\ell^u \subseteq \MM_\ell$ and (at least) $n-k$ times $\overline\MM_\ell^{uz} \subseteq \MM_\ell$. From Proposition~\ref{prop:linear-convergence}, we hence infer that
	\begin{align*}
	\eta_{\ell + n}(u_{\ell + n})^2
	&\le (\Crel^2 + \gamma)\gamma^{-1} \, \qctr^k \, \eta_\ell(u_\ell)^2
	\intertext{as well as}
	\big[ \, \eta_{\ell + n}(u_{\ell + n})^2 + \zeta_{\ell + n}(z_{\ell + n}[u_{\ell + n}])^2 \big] 
	&\le (2 \, \Crel^2 + \gamma)\gamma^{-1} \,  \qctr^{n-k} \, \big[ \, \eta_\ell(u_\ell)^2 + \zeta_\ell(z_\ell[u_\ell])^2\, \big].
	\end{align*}
	Multiplying these two estimates, we conclude the proof with
	$\qlin = \qctr^{1/2}$ and $\Clin = (2 \, \Crel^2 + \gamma)/\gamma$.
\end{proof}

\subsection{Optimal rates}
Linear convergence, together with the following lemma, finally proves optimal rates for Algorithm~\ref{algorithm}.
%

\begin{lemma}\label{step1:optimal}
Suppose \eqref{assumption:stab}--\eqref{assumption:drel}.
For all $0<\theta<\theta_{\rm opt} =(1+\Cstab^2\Cdrel^2)^{-1}$,
there exists $\Ctwo>0$ such that the following holds: For all $\TT_\coarse \in \T$, there exists some $\RR_\coarse \subseteq \TT_\coarse$ such that 
for all $s,t>0$ with $\norm{u}{\mathbb{A}_s} + \norm{z[u]}{\mathbb{A}_t} < \infty$ and $\alpha = \min\{2s,s+t\}$, it holds that
\begin{align}\label{eq:new1}
\begin{split}
 \#\RR_\coarse
 \le 2\,\big[\,\Ctwo\,\norm{u}{\mathbb{A}_s}(\norm{u}{\mathbb{A}_s}+\norm{z[u]}{\mathbb{A}_t})\,\big]^{1/\alpha}\,
\Big(\eta_\coarse(u_\coarse)\big[ \, \eta_\coarse(u_\coarse)^2 + \zeta_\coarse(z_\coarse[u_\coarse])^2 \, \big]^{1/2}\Big)^{-1/\alpha}
\end{split}
\end{align}
as well as the D\"orfler marking~\eqref{eq:doerfler}, i.e.,
\begin{align}\label{eq:new2}
\begin{split}
 \theta\eta_\coarse(u_\coarse)^2 
 &\le \eta_\coarse(\#\RR_\coarse, u_\coarse)^2
 \quad\text{or}\quad
 \\
 \theta\big[ \, \eta_\coarse(u_\coarse)^2 + \zeta_\coarse(z_\coarse[u_\coarse])^2 \, \big] 
 &\le  \, \eta_\coarse(\#\RR_\coarse, u_\coarse)^2 + \zeta_\coarse(\#\RR_\coarse, z_\coarse[u_\coarse])^2.
\end{split}
\end{align}
The constant $\Ctwo$ depends only on $\theta$ and \eqref{assumption:stab}--\eqref{assumption:drel}.
\end{lemma}%

\begin{proof}
Adopt the notation of Lemma~\ref{lemma:doerfler}.
According to Lemma~\ref{lemma:stab:z}, stability~\eqref{assumption:stab} and reliability~\eqref{assumption:rel}, it holds that
\begin{equation*}
	\zeta_\coarse(z_\coarse[u_\coarse])
	\leq
	C \big[  \eta_\coarse(u_\coarse) + \zeta_\coarse(z_\coarse[u])  \big]
	\text{ and }
	\zeta_\coarse(z_\coarse[u])
	\leq
	C \big[  \eta_\coarse(u_\coarse) + \zeta_\coarse(z_\coarse[u_\coarse])  \big]
\end{equation*}
with $C := \max\{ 1, \Cstab \Crel C_1 C_2 \}$.
The quasi-monotonicity of the estimators (Lemma~\ref{lemma:quasi-monotonicity}) and $\big[ \, \eta_\fine(u_\fine)^2 + \zeta_\fine(z_\fine[u_\fine])^2 \, \big]^{1/2} \le (C+1) \big[ \, \eta_\fine(u_\fine) + \zeta_\fine(z_\fine[u]) \, \big]$ yield that
\begin{align*}
 \eps &:= (C+1)^{-1}\Cmon^{-1}\kappa_{\rm opt}\,\eta_\coarse(u_\coarse)\big[ \, \eta_\coarse(u_\coarse)^2 + \zeta_\coarse(z_\coarse[u_\coarse])^2 \, \big]^{1/2} \\
 &\le (C+1)^{-1}\kappa_{\rm opt}\,\eta_0(u_0)\big[ \, \eta_0(u_0)^2 + \zeta_0(z_0[u_0])^2 \, \big]^{1/2}
 \\&
 < \eta_0(u_0)(\eta_0(u_0)+\zeta_0(z_0[u]))
 \le \norm{u}{\mathbb{A}_s}(\norm{u}{\mathbb{A}_s}+\norm{z[u]}{\mathbb{A}_t}) < \infty.
\end{align*}
Choose the minimal $N\in\N_0$ such that 
\begin{align*}
 \norm{u}{\mathbb{A}_s}(\norm{u}{\mathbb{A}_s}+\norm{z[u]}{\mathbb{A}_t}) \le\eps\,(N+1)^{\alpha}.
\end{align*}
From the choice of $\varepsilon$ and the previous estimate, it follows that $N > 0$.
Choose $\TT_{\eps_1},\TT_{\eps_2}\in\T_N$ with 
$\eta_{\eps_1}(u_{\eps_1})=\min_{\TT_\fine\in\T_N}\eta_\fine(u_\fine)$ and
$\zeta_{\eps_2}(z_{\eps_2}[u]) = \min_{\TT_\fine\in\T_N} \zeta_\fine(z_\fine[u])$.
Define $\TT_\eps:=\TT_{\eps_1}\oplus\TT_{\eps_2}$
and $\TT_\fine:=\TT_\eps\oplus\TT_\coarse$.
Then, Lemma~\ref{lemma:quasi-monotonicity}, the definition of the 
approximation classes, and the choice of $N$  and $\alpha = \min\{2s,s+t\}$ give that
\begin{align*}
 &\eta_\fine(u_\fine)\big[ \, \eta_\fine(u_\fine)^2 + \zeta_\fine(z_\fine[u_\fine])^2 \, \big]^{1/2}
 \le  \Cmon \eta_{\eps_1}(u_{\eps_1})\big[ \, \eta_{\eps_1}(u_{\eps_1})^2 + \zeta_{\eps_2}(z_{\eps_2}[u_{\eps_2}])^2 \, \big]^{1/2}
 \\&\quad
 \le (C+1) \Cmon \eta_{\eps_1}(u_{\eps_1}) \, \big[\,\eta_{\eps_1}(u_{\eps_1})+\zeta_{\eps_2}(z_{\eps_2}[u]) \, \big]
 \\&\quad
 \le (C+1) \Cmon\big((N+1)^{-(2s)}\norm{u}{\mathbb{A}_s}^2 + (N+1)^{-(s+t)}\norm{u}{\mathbb{A}_s}\norm{z[u]}{\mathbb{A}_t}\big)
 \\&\quad
 \le (C+1) \Cmon(N+1)^{-\alpha}\norm{u}{\mathbb{A}_s}(\norm{u}{\mathbb{A}_s}+\norm{z[u]}{\mathbb{A}_t})
 \\&\quad
 \le (C+1) \Cmon\eps  = \kappa_{\rm opt}\,\eta_\coarse(u_\coarse)\big[ \, \eta_\coarse(u_\coarse)^2 + \zeta_\coarse(z_\coarse[u_\coarse])^2 \, \big]^{1/2}.
\end{align*}
This implies that $\eta_\fine(u_\fine)^2 \le \kappa_{\rm opt}\,\eta_\coarse(u_\coarse)^2$ or
$\big[ \, \eta_\fine(u_\fine)^2 + \zeta_\fine(z_\fine[u_\fine])^2 \, \big] \le \kappa_{\rm opt}\,\big[ \, \eta_\coarse(u_\coarse)^2 + \zeta_\coarse(z_\coarse[u_\coarse])^2 \, \big]$. 
Lemma~\ref{lemma:doerfler} hence proves~\eqref{eq:new2} with $\RR_\coarse := \TT_\coarse \backslash \TT_\fine$. It remains to derive~\eqref{eq:new1}. To that end, define 
\begin{align*}
 \widetilde{C} &:= \big[\,\norm{u}{\mathbb{A}_s}(\norm{u}{\mathbb{A}_s}+\norm{z[u]}{\mathbb{A}_t})\,\Cmon\kappa_{\rm opt}^{-1}\,\big]^{1/\alpha}.
\end{align*}
Then, minimality of $N \in \N_0$ and $N > 0$ yield that
\begin{align*}
 N < \big[\,\norm{u}{\mathbb{A}_s}(\norm{u}{\mathbb{A}_s}+\norm{z[u]}{\mathbb{A}_t})\,\big]^{1/\alpha}\eps^{-1/\alpha}
 = \widetilde{C} \,\Big(\eta_\coarse(u_\coarse)\big[ \, \eta_\coarse(u_\coarse)^2 + \zeta_\coarse(z_\coarse[u_\coarse])^2 \, \big]^{1/2} \Big)^{-1/\alpha}.
\end{align*}
According to the choice of $\TT_\fine$ and $\RR_\coarse$,
the overlay estimate~\eqref{eq:mesh-overlay} yields that
\begin{align}\label{eq:new:step2}
\begin{split}
&\#\RR_\coarse = \#(\TT_\coarse \backslash \TT_\fine) 
\stackrel{\eqref{eq:mesh-sons}}\le \#\TT_\fine - \#\TT_\coarse
\stackrel{\eqref{eq:mesh-overlay}}\le\#\TT_\eps - \#\TT_0
\stackrel{\eqref{eq:mesh-overlay}}\le\#\TT_{\eps_1} + \#\TT_{\eps_2} - 2\,\#\TT_0
\le 2N 
\\& \qquad 
< 2 \widetilde{C} \, \Big( \eta_\coarse(u_\coarse)\big[ \, \eta_\coarse(u_\coarse)^2 + \zeta_\coarse(z_\coarse[u_\coarse])^2 \, \big]^{1/2} \Big)^{-1/\alpha}.
\end{split}
\end{align}
Overall, we conclude~\eqref{eq:new1} with $\Ctwo = \Cmon / \kappa_{\rm opt}$.
\end{proof}

\begin{proof}[\bfseries Proof of Theorem~\ref{theorem:main}\rm\bf(ii)]
According to~\eqref{eq:new2} of Lemma~\ref{step1:optimal} and the marking strategy in Algorithm~\ref{algorithm}, for all $j\in\N_0$, it holds that
\begin{align*}
\begin{split}
\#\MM_j &\le 2 \, \min\{\#\overline\MM_j^u \,,\, \#\overline\MM_j^{uz}\}
\le 2 \Cmark \, \#\RR_j.
\end{split}
\end{align*}
With $\alpha = \min\{2s,s+t\}>0$, estimate~\eqref{eq:new1} of Lemma~\ref{step1:optimal} implies that
\begin{align*}
 \#\MM_j \le 4 \, \Cmark \,\big[\,\Ctwo\,\norm{u}{\mathbb{A}_s}(\norm{u}{\mathbb{A}_s}+\norm{z}{\mathbb{A}_t})\,\big]^{1/\alpha}\,
 \Big(\eta_j(u_j)\big[ \, \eta_j(u_j)^2 + \zeta_j(z_j[u_j])^2 \, \big]^{1/2} \Big)^{-1/\alpha}.
\end{align*}
With the mesh-closure estimate~\eqref{eq:mesh-closure}, we obtain that
\begin{align*}
\#\TT_\ell-\#\TT_0
&\stackrel{\eqref{eq:mesh-closure}}\le\Cnvb \sum_{j=0}^{\ell-1}\#\MM_j
=
\Cnvb \Big( \sum_{j=0}^{\ell_0-1} \#\MM_j + \sum_{j=\ell_0}^{\ell-1} \#\MM_j \Big).
\end{align*}
Using that $\#\MM_j \leq \#\TT_{j+1} - \#\TT_j$, we get that
\begin{equation}\label{eq:4711}
\begin{split}
	\#\TT_\ell-\#\TT_0
	&\leq
	\Cnvb \Big( \#\TT_{\ell_0}-\#\TT_0 + \sum_{j=\ell_0}^{\ell-1} \#\MM_j \Big)
	\leq
	\Cnvb \big( \#\TT_{\ell_0} - \#\TT_0 + 1 \big) \sum_{j=\ell_0}^{\ell-1} \#\MM_j \\
	& \qquad \lesssim
	\sum_{j=\ell_0}^{\ell-1} \Big(\eta_j(u_j)\big[ \, \eta_\coarse(u_j)^2 + \zeta_j(z_j[u_j])^2 \, \big]^{1/2} \Big)^{-1/\alpha}.
\end{split}
\end{equation}%
Linear convergence~\eqref{eq:linear} implies that
\begin{align*}
 \eta_\ell(u_\ell)\big[ \, \eta_\ell(u_\ell)^2 + \zeta_\ell(z_\ell[u_\ell])^2 \, \big]^{1/2}
 \le\Clin\, \qlin^{\ell-j}\eta_j(u_j)\big[ \, \eta_j(u_j)^2 + \zeta_j(z_j[u_j])^2 \, \big]^{1/2}
\end{align*}
for all $0\le j\le\ell$ and hence
\begin{align*}
 &\Big( \eta_j(u_j)\big[ \, \eta_\coarse(u_j)^2 + \zeta_j(z_j[u_j])^2 \, \big]^{1/2} \Big)^{-1/\alpha}\\
 &\hspace{100pt}\le \Clin^{1/\alpha} \qlin^{(\ell-j)/\alpha}
 \Big( \eta_\ell(u_\ell)\big[ \, \eta_\ell(u_\ell)^2 + \zeta_\ell(z_\ell[u_\ell])^2 \, \big]^{1/2} \Big)^{-1/\alpha}.
\end{align*}
With $0<q:=\qlin^{1/\alpha}<1$, the geometric series applies and yields that
\begin{align*}
 &\sum_{j=\ell_0}^{\ell-1} \Big( \eta_j(u_j)\big[ \, \eta_\coarse(u_j)^2 + \zeta_j(z_j[u_j])^2 \, \big]^{1/2} \Big)^{-1/\alpha}\\
 &\hspace{100pt}\le \Clin^{1/\alpha}  \Big( \eta_\ell(u_\ell)\big[ \, \eta_\ell(u_\ell)^2 + \zeta_\ell(z_\ell[u_\ell])^2 \, \big]^{1/2} \Big)^{-1/\alpha}\,
 \sum_{j=\ell_0}^{\ell-1} q^{\ell-j}
 \\&\hspace{100pt}
 \le \frac{\Clin^{1/\alpha}}{1-\qlin^{1/\alpha}}\, \Big( \eta_\ell(u_\ell)\big[ \, \eta_\ell(u_\ell)^2 + \zeta_\ell(z_\ell[u_\ell])^2 \, \big]^{1/2} \Big)^{-1/\alpha}. 
\end{align*}
Combining this with~\eqref{eq:4711}, we obtain that
\begin{align*}
 \#\TT_\ell-\#\TT_0
 \le 4\, \frac{\Cnvb\Cmark}{1-\qlin^{1/\alpha}} \big( \#\TT_{\ell_0} - \#\TT_0 + 1 \big) \,
 &\big[\,\Clin \Ctwo\,\norm{u}{\mathbb{A}_s}(\norm{u}{\mathbb{A}_s}+\norm{z}{\mathbb{A}_t})\,\big]^{1/\alpha}
 \\& \quad
 \,\Big(\eta_\ell(u_\ell)\big[ \, \eta_\ell(u_\ell)^2 + \zeta_\ell(z_\ell[u_\ell])^2 \, \big]^{1/2} \Big)^{-1/\alpha}.
\end{align*}
Altogether, we conclude~\eqref{eq:optimal} with $\expandafter\widetilde\Copt := \max\{ \Clin \Ctwo \,, 4 \, \Cnvb\Cmark \}$ and $\Copt = \expandafter\widetilde\Copt^{1+\alpha} \big( \#\TT_{\ell_0} - \#\TT_0 + 1 \big)^{1+\alpha} / (1-\qlin^{1/\alpha})^{\alpha}$.
\end{proof}

\section{Proof of Theorem~\lowercase{\ref{theorem:combined}}}\label{sec:proof:combined}

In contrast to the corresponding results for Algorithm~\ref{algorithm}, the proof of Theorem~\ref{theorem:combined} (for Algorithm~\ref{algorithm:combined}) follows essentially from the abstract setting of~\cite{axioms}.

\begin{proof}[\bfseries Proof of Theorem~\ref{theorem:combined}]	
	{\rm (i)} Note that \eqref{eq:doerfler:dual} coincides with \eqref{eq:doerfler:uz}.
	Hence, Proposition~\ref{prop:linear-convergence} can be applied and results in
	\begin{equation*}
		\big[ \,
			\Delta_{\ell+n}^u + \Delta_{\ell+n}^z
		\, \big]
		\leq
		\qctr \, \big[ \,
			\Delta_{\ell}^u + \Delta_{\ell}^z
		\, \big]
		\quad \text{for all } \ell, n \in \N_0 \text{ with } \ell \geq \ell_0.
	\end{equation*}
	For $n=1$, we conclude contraction and hence linear convergence
	\begin{align*}
		&\big[ \,
			\eta_k(u_k)^2 + \zeta_k(z_k[u_k])^2
		\, \big]
		\simeq
		\, \big[ \,
			\Delta_{k}^u + \Delta_{k}^z
		\, \big] \\
		&\qquad \qquad \leq
		\qctr^{k-\ell} \, \big[ \,
			\Delta_{\ell}^u + \Delta_{\ell}^z
		\, \big]
		\simeq
		\qctr^{k-\ell} \, \big[ \,
			\eta_\ell(u_\ell)^2 + \zeta_\ell(z_\ell[u_\ell])^2
		\, \big]
		\quad \text{for all } k \geq \ell \geq \ell_0.
	\end{align*}
	
	{\rm (ii)} Since Proposition~\ref{prop:axioms-product-estimator} shows that there hold \eqref{assumption:stab}--\eqref{assumption:drel} even for the combined estimator, \cite[Theorem~4.1(ii)]{axioms} guarantees convergence with optimal rates according to the approximation class
	\begin{equation*}
		\norm{(u,z[u])}{\mathbb{A}_\beta} 
		:= \sup_{N\in\N_0} \Big( 
			(N+1)^\beta \min_{\TT_\coarse\in\T_N} \big[ \eta_\coarse(u)^2 + \zeta_\coarse(z_\coarse[u])^2 \big]^{1/2}
		\Big)
		\in \R_{\geq 0} \cup \{ \infty \}
	\end{equation*}
	with $\beta > 0$.
	In particular, there exists a constant $\expandafter\widetilde\Copt > 0$ such that
	\begin{equation}\label{eq:thm:optimal-dual:axioms}
		\sup_{\ell \in \N_0}
		\frac{[ \eta_\ell(u_\ell)^2 + \zeta_\ell(z_\ell[u_\ell])^2]^{1/2}}{(\#\TT_\ell - \#\TT_0 + 1)^{-\beta}}
		\leq
		\expandafter\widetilde\Copt \, \norm{(u,z[u])}{\mathbb{A}_\beta}
	\end{equation}
	for all $\beta > 0$.
	For $N \in \N_0$ choose $\TT_{\coarse^u}, \TT_{\coarse^z} \in \T_N$ such that
	\begin{equation}\label{eq:minimality}
		\eta_{\coarse^u}(u_{\coarse^u})
		=
		\min_{\TT_\star \in \T_N} \eta_\star(u_\star),
		\quad
		\zeta_{\coarse^z}(z_{\coarse^z}[u])
		=
		\min_{\TT_\star \in \T_N} \zeta_\star(z_\star[u]).
	\end{equation}
	Then, for the overlay $\TT_\coarse := \TT_{\coarse^u} \oplus \TT_{\coarse^z}$, it holds that
	\begin{equation*}
		\#\TT_\coarse - \#\TT_0
		\leq
		\#\TT_{\coarse^u} + \#\TT_{\coarse^z} - 2 \#\TT_0
		\leq
		2 N
	\end{equation*}
	and thus $\TT_\coarse \in \T_{2N}$.
	From the minimality assumption in \eqref{eq:minimality}, we infer that
	\begin{align*}
		(2N+1)^\beta \big[ \,
			\eta_\coarse(u_\coarse)^2 &+ \zeta_\coarse(z_\coarse[u])^2
		\, \big]^{1/2}
		\leq
		2^\beta \big[ \,
			(N+1)^\beta \eta_\coarse(u_\coarse) + (N+1)^\beta \zeta_\coarse(z_\coarse[u])
		\, \big] \\
		&\stackrel{\mathclap{\eqref{eq:quasi-monotonicity}}}{\leq}
		2^\beta \, \big[ \,
			\Cmon (N+1)^\beta \eta_{\coarse^u}(u_{\coarse^u}) + \Cmon (N+1)^\beta \zeta_{\coarse^z}(z_{\coarse^z}[u])
		\, \big] \\
		&\stackrel{\mathclap{\eqref{eq:minimality}}}{\leq}
		2^\beta \Cmon \, \big[ \,
			\norm{u}{\mathbb{A}_\beta} + \norm{z}{\mathbb{A}_\beta}
		\, \big].
	\end{align*}
	Hence, we have $\norm{(u,z[u])}{\mathbb{A}_\beta} \lesssim \norm{u}{\mathbb{A}_s} + \norm{z}{\mathbb{A}_t}$ for $\beta \leq \min\{s,t\}$.
	From this and \eqref{eq:thm:optimal-dual:axioms}, we obtain \eqref{eq:combined:optimal} by squaring, thus doubling the rate, i.e.\ $\beta=2\alpha$.
\end{proof}

%

{
	\renewcommand{\section}[3][]{\vskip4mm\begin{center}\bf\normalsize R\small EFERENCES\normalsize\end{center}\vskip2mm}
	
	\bibliographystyle{alpha}
	\bibliography{literature}
}


\clearpage


%
%



\end{document}